\documentclass[a4paper,11pt]{article}  
\usepackage[english]{babel}
\usepackage{mathptmx}
\usepackage{berasans}
\usepackage{beramono}

\usepackage[T1]{fontenc}
\usepackage[latin9]{inputenc}
\usepackage{geometry}
\usepackage{color}
\usepackage{babel}
\usepackage{float}
\usepackage{amsthm}
\usepackage{amsmath}
\usepackage{amssymb}
\usepackage{graphicx}
\usepackage{esint}
\usepackage{amsfonts}
\usepackage{graphicx}
\usepackage{amsmath, bbm, amssymb}
\usepackage{amsthm}
\usepackage{amsmath}
\usepackage{amssymb}
\usepackage{enumerate}
\usepackage{color, soul}

\usepackage{amsmath,amsfonts}
\usepackage{graphicx}
\usepackage[colorlinks=true, allcolors=blue]{hyperref}
\usepackage[authoryear]{natbib}

\theoremstyle{definition}
\newtheorem{defn}{Definition}
\theoremstyle{plain}
\newtheorem{lem}[defn]{Lemma}
\theoremstyle{definition}
\newtheorem{example}[defn]{Example}
\theoremstyle{plain}
\newtheorem{prop}[defn]{Proposition}
\theoremstyle{plain}
\newtheorem{thm}[defn]{Theorem}
\theoremstyle{definition}
\newtheorem{rem}[defn]{Remark}
\theoremstyle{plain}

\ifx\proof\undefined\
\newenvironment{proof}[1][\proofname]{\par
	\normalfont\topsep6\p@\@plus6\p@\relax
	\trivlist
	\itemindent\parindent
	\item[\hskip\labelsep
	\scshape
	#1]\ignorespaces
}{%
	\endtrivlist\@endpefalse
}
\providecommand{\proofname}{Proof}
\fi

\usepackage{indentfirst}

\makeatother

\begin{document}
	
	\title{\textbf{Quasi-copulas with a prescribed  sub-track section}}

	\author{\textit{Jiehua Xie$^{\rm a}$,\ \ 
			Wei Zou$^{\rm b,c}$\thanks{Corresponding
				author.  \textit{Email-address:} zouwei@juwp.edu.cn (Wei Zou).}}
	\\
	\small  \textit{$^{\rm a}$School of Business Administration, Jiangxi University of Water Resources and Electric Power, Jiangxi 330099, P.R.China
	}
		\\
		\small  \textit{$^{\rm b}$School of Science, Jiangxi University of Water Resources and Electric Power, Jiangxi 330099, P.R.China} 
		\\ 
		\small  \textit{$^{\rm c}$The Key Laboratory of Engineering Mathematics and Advanced Computing, }
		\\ 
		\small  \textit{Jiangxi University of Water Resources and Electric Power, Jiangxi 330099, P.R.China} }

	\date{October 2026}
	\maketitle
	\begin{abstract}
		In this paper, we consider the quasi-copulas and copulas with a prescribed sub-track section, i.e., with given values along the curve segment connecting the points $(x_0,0)$ and $(1,1-x_0)$ for $0<x_0<1$. By employing the technique based on the $W$-ordinal sum, we provide  methods for constructing quasi-copulas and copulas whose sub-track section coincides with a given one. Then, we further investigate construction methods for symmetric quasi-copulas and copulas with a prescribed sub-track section. Furthermore, we determine  the best-possible bounds for the class of quasi-copulas sharing a common sub-track section, and verify that the lower bound obtained is also the best-possible lower bound for the class of copulas with a common sub-track section. Finally, we examine several  properties of the bound (quasi-)copulas with a prescribed sub-track section, including their comparison with bound (quasi-)copulas with a given track section or a sub-diagonal section, as well as their statistical characteristic and concordance order.
	\end{abstract}
	
	\section{Introduction}
	
Quasi-copulas and copulas are fundamental tools in a wide range of applications \citep{McNeil-et-al.(2015),2022Multivariate,2023Asymptotic}. A copula is
	a joint distribution function whose univariate margins are uniformly distributed on [0, 1].
In view of  Sklar's Theorem \citep{sklar1959fonctions}, each joint distribution function of a random vector can be expressed in terms of its marginal
	distribution functions and an associated copula. The concept of a quasi-copula is more general than
	that of a copula.  The quasi-copula 
	was  originally introduced in the field of probability \citep{1993On} and later  characterized in \cite{2010On}. Both copulas and quasi-copulas are 1-Lipschitz aggregation operators with neutral element 1 \citep{H2008On}, and they
	are also employed in aggregation processes due to their stability property \citep{Genest1999A}. For more detailed introduction about quasi-copulas and copulas, we refer to   \cite{durante2015principles},  \cite{Joe(1997)} and \cite{Nelsen(2006)}. 
	
	In the theoretical research on quasi-copulas and copulas, considerable attention has been devoted to the construction of quasi-copulas and copulas with prescribed values along specified
	sections in the unit square, as well as to the determination of best-possible bounds for the resulting function classes \citep{Fredricks(2002),2001BestBounds}. In particular, since the tail dependence coefficients of a copula are completely determined by its \textit{diagonal section},   the problems of constructing quasi-copulas and copulas with a prescribed diagonal section and deriving the corresponding best-possible bounds have been extensively investigated  \citep{2024Exact,Butucea2015,2009New,2008On,2008Onsoft}.

	More recently, research has increasingly focused on the more general setting, that is,  bivariate  quasi-copulas and copulas  with a prescribed  \textit{track section} (also referred to as \textit{curvilinear section}) along a curve  $\phi$ on the unit square, where $\phi:[0,1]\rightarrow[0,1]$ is  a strictly increasing and continuous function satisfying $\phi(0)=0$ and $\phi(1)=1$. 
	For instance, \citet{Fernandez2017Copulas} established existence conditions for  copulas with a prescribed track section $\Gamma_{C,\phi}(u)=C\big(u,\phi(u)\big)$, $u\in[0,1]$, where $C$ is a bivariate copula. Subsequently, \citet{2021Curved},  \citet{2024curve} and \citet{2023Curvilinear}  introduced the curvilinear patchwork operation as a tool for  constructing  more copulas and quasi-copulas sharing a common track section from the existing ones, while \citet{2026on} and  \citet{2026con} addressed the construction  of absolutely continuous copulas with a prescribed  track section. Alongside these contributions on construction methods, the problem of identifying best-possible bounds for classes of copulas and quasi-copulas with a prescribed track section has also attracted considerable interest. For a given track section  $\Gamma_{\phi}:[0,1]\rightarrow[0,1]$, 
	it has been established that the \textit{Bertino copula with the given curvilinear section}, defined as 	\begin{equation}\label{eq:Bertino copula+in}B_{\Gamma_{\phi}}(u,v)
		=\left\{\begin{array}{ll}
			v-\min\{\phi(t)-\Gamma_{\phi}(t)| t\in[\phi^{-1}(v),u]\},& \ v<\phi(u),\\
			u-\min\{t-\Gamma_{\phi}(t)|t\in[u,\phi^{-1}(v)] \},&  \ v\ge \phi(u), 
		\end{array}\right.
	\end{equation} 
	is the best-possible lower bound for the class of copulas sharing the common track section  $\Gamma_{\phi}$ \citep{De2019On}, and it   simultaneously serves as the best-possible lower bound for the corresponding class of quasi-copulas   \citep{2022Best}. The best-possible upper bound for the class of  quasi-copulas with the common track section $\Gamma_{\phi}$ is given by the  \textit{$\Gamma_{\phi}$-curvilinear quasi-copula} $A_{\Gamma_{\phi}}$ \citep{2022Best}, expressed as 	\begin{equation}\label{eq:A copula}
	A_{\Gamma_{\phi}}(u,v)	=\left\{\begin{array}{ll}
			\min\big\{v,u-\max\{t-\Gamma_{\phi}(t)|t\in[\phi^{-1}(v),u]\}\big\},& \ v<\phi(u),\\
			\min\big\{u,v-\max\{\phi(t)-\Gamma_{\phi}(t)|t\in[u,\phi^{-1}(v)]\}\big\},&  \ v\ge \phi(u).
		\end{array}\right.
	\end{equation} 
Furthermore, the \textit{$\Gamma_{\phi}$-curvilinear copula} 	$K_{\Gamma_{\phi}}$ defined by
\begin{equation}\label{eq:K copula}
	K_{\Gamma_{\phi}}(u,v)=\min\{u,v,\frac{\Gamma_{\phi}(u)+\Gamma_{\phi}\big(\phi^{-1}(v)\big)}{2}\},
\end{equation}
is a copula if and only if the functions $2t-\Gamma_{\phi}(t)$ and $2t-\Gamma_{\phi}\big(\phi^{-1}(t)\big)$ are nondecreasing
on [0, 1], and it is 
the best-possible upper bound for the class of copulas sharing the common track section $\Gamma_{\phi}$ that are symmetric with respect to the curve $\phi$ on a specified region of the unit square \citep{2022Best}.  A complete characterization of these maximal copulas with the given track
section $\Gamma_{\phi}$ was recently proposed by \citet{2026a}. 

The present work focuses on  quasi-copulas and copulas with a more general type of section, that is, a given \textit{sub-track section} $\Gamma_{\phi_{x_0}}:[0,1-x_0]\rightarrow[0,1-x_0]$, corresponding to given values  along the curve segment joining the points $(x_0,0)$ and $(1,1-x_0)$ for $0\le x_0<1$. Notice that the sub-track section $\Gamma_{\phi_{x_0}}$ in the case $x_0=0$ reduces to the standard track section.  Given values on the more general section $\Gamma_{\phi_{x_0}}$, the first natural problem  that arises is the following one: Does there exist a copula or a quasi-copula $Q$ whose sub-track section 
coincides with
$\Gamma_{\phi_{x_0}}$?  A second question is then: What
are the best-possible bounds for the class of all the quasi-copulas sharing this sub-track section $\Gamma_{\phi_{x_0}}$? While these questions have been fully resolved for the particular case of $x_0 = 0$ (i.e., for 
quasi-copulas and copulas with a prescribed track section), the general case $0<x_0<1$ remains open.  As will be seen, we  answer both questions affirmatively by constructing some 
quasi-copulas and copulas with the prescribed sub-track section and subsequently deriving the corresponding lower and upper bounds.

In this paper, we  employ the technique based on the $W$-ordinal sum to develop construction methods for copulas and quasi-copulas with a prescribed sub-track section $\Gamma_{\phi_{x_0}}$. Using the functions  $B_{\Gamma_{\phi}}$, $A_{\Gamma_{\phi}}$ and $K_{\Gamma_{\phi}}$ presented in \eqref{eq:Bertino copula+in}-\eqref{eq:K copula}, we construct three families of quasi-copulas and copulas whose  sub-track section coincides with the given one $\Gamma_{\phi_{x_0}}$. We further extend these constructions to the symmetric setting, obtaining three families of symmetric quasi-copulas and one family of symmetric copulas with the prescribed sub-track section $\Gamma_{\phi_{x_0}}$. Furthermore, we establish the best-possible bounds for the class of quasi-copulas sharing a common sub-track section, and show that the lower bound obtained is also best-possible lower bound for the corresponding class   of copulas. Finally, some properties of the bound (quasi-)copulas with a prescribed sub-track section are also discussed, including  comparisons with bound (quasi-)copulas with a given track section or a sub-diagonal section, as well as their statistical characteristic and  concordance order.

	The remainder of this paper is organized as follows.  Section 2 reviews definitions and fundamental properties concerning  quasi-copulas and copulas, and introduces the preliminaries on sub-track sections, which will be used throughout. In Section 3, we discuss  construction methods for  quasi-copulas and copulas whose  sub-track section coincides with a given one.  Section 4 presents families of (quasi-)copulas  and symmetric quasi-copulas with a prescribed sub-track section, while Section 5 is devoted to the construction of  symmetric copulas with such a section. In  Section 6, we establish the best-possible bounds for the class of quasi-copulas sharing a common sub-track section.  
Section 7 investigates some properties of these  bound (quasi-)copulas, and 
	conclusions are drawn in Section 8.

	\section{Preliminaries}
	
	In this section, we review  definitions and some fundamental properties concerning quasi-copulas and copulas. We also introduce the notations of track sections and sub-track sections. These definitions and notions are preliminaries for investigating   quasi-copulas and copulas  with a prescribed sub-track section. 
	
	\subsection{Quasi-copula and copula}
	
	This subsection presents the definitions of bivariate quasi-copulas and copulas, and briefly reviews their fundamental properties.
	
	We begin by presenting the notion of the bivariate quasi-copula.
	
	\begin{defn}\citep{Nelsen(2006)}\label{df:qcopula}
		A two-dimensional function $Q$: $[0,1]^2\rightarrow[0,1]$ is said to be a \textit{bivariate quasi-copula} (for simplicity, a \textit{quasi-copula}) if it  satisfies the following conditions:
		\begin{itemize}
			\item[(Q1)] (Boundary conditions) $Q(0,t)=Q(t,0)=0$ and $Q(1,t)=Q(t,1)=t$ for all $t\in[0,1]$;
			\item[(Q2)] (Nondecreasing condition) $Q$ is nondecreasing in each argument, that is, for every fixed $v\in[0,1]$,  the mappings $u\mapsto Q(u, v)$ and $u \mapsto Q(v, u)$ are both nondecreasing; 
			\item[(Q3)] (Lipschitz condition) for
			all $u_1$, $u_2$, $v_1$ and $v_2$ in $[0, 1]$,
			
			\begin{equation}\label{Lp}
				|Q(u_1,v_1)-Q(u_2,v_2)|\le |u_1-u_2|+|v_1-v_2|. 
			\end{equation}	 
		\end{itemize}
	\end{defn}	

A copula is a special sub-class of quasi-copulas, characterized by additional conditions beyond those required for quasi-copulas.
	
	\begin{defn}\citep{Nelsen(2006)}\label{df:copula}
		A two-dimensional function $C$: $[0,1]^2\rightarrow[0,1]$ is said to be a \textit{bivariate copula} (for simplicity, a   \textit{copula}) if it satisfies the following conditions:
		\begin{itemize}
			\item[(C1)] (Boundary conditions) $C(0,t)=C(t,0)=0$ and $C(1,t)=C(t,1)=t$ for all $t\in[0,1]$;
			\item[(C2)] (2-increasing condition) for all $u_1, u_2, v_1$ and $v_2$ in $[0, 1]$  with  $u_1\le u_2$ and $v_1\le v_2$, \begin{equation}\label{qe:V}V_C([u_1,u_2]\times[v_1,v_2]):=C(u_2,v_2)-C(u_2,v_1)-C(u_1,v_2)+C(u_1,v_1)\ge 0.\end{equation}
			Here, $V_C([u_1,u_2]\times[v_1,v_2])$ is referred to as the $C$-\textit{volume} of the rectangle $[u_1,u_2]\times[v_1,v_2]$.	 
		\end{itemize}
	\end{defn}
\noindent	 

It follows directly from Definitions \ref{df:qcopula} and \ref{df:copula} that every copula $C$  satisfies the nondecreasing condition (Q2) and the Lipschitz condition (Q3). Therefore, every copula is also a quasi-copula. The converse, however, does not hold in general. There exist quasi-copulas that fail to satisfy the 2-increasing condition (C2), and such quasi-copulas are 
called \textit{proper} quasi-copulas. A copula or a  quasi-copula $Q$ is said to be symmetric if
$Q(u, v) = Q(v, u)$ for all $(u, v)\in[0, 1]^2$.

Recall that a copula is equivalently characterized as a joint distribution function whose univariate margins are uniformly distributed on $[0, 1]$. 
	Classical examples of bivariate copulas include the product copula $\Pi(u,v)=uv$, the Fr\'{e}chet-Hoeffding upper bound $M(u,v)=\min\{u,v\}$ and the Fr\'{e}chet-Hoeffding  lower bound  $W(u,v)=\max\{u+v-1, 0\}$, for $(u,v)\in [0,1]^2$, which correspond respectively to the notions of independence, comonotonicity, and  counter-monotonicity between two random variables.  Notice easily that all the three copulas are symmetric. Moreover, for any copula $C$, the following bounds are well-known:  
	$$W(u,v)\le C(u,v)\le M(u,v),\ \ (u,v)\in[0,1]^2.$$
	For more examples and properties of quasi-copulas and  copulas, please see  \cite{durante2015principles} and \cite{Nelsen(2006)}.

	\subsection{Track section and sub-track section}
	
	In this subsection, we present the curves that will be considered throughout this paper, and then introduce the notions of track section  and sub-track section  along prescribed curves on the unit square. 
	
	Let $\phi$: $[0,1]\rightarrow$ $[0,1]$ be  a strictly increasing continuous function satisfying  $\phi(0)=0$ and  $\phi(1)=1$. Analogously, for any  $0\le x_0< 1$, let $\phi_{x_0}$: $[0,1-x_0]\rightarrow$ $[0,1-x_0]$ be a strictly increasing continuous function  with $\phi_{x_0}(0)=0$ and  $\phi_{x_0}(1-x_0)=1-x_0$. Since $\phi_{x_0}$ is strictly increasing and continuous, its inverse $\phi^{-1}_{x_0}$ is well defined on $[0,1-x_0]$. 
	 For notational convenience, throughout the remainder of this paper, the curve  $\{\big(t,\phi(t)\big)|t\in[0,1]\}$ on the unit square will be referred to as the curve  $\phi$, and  the curve $\{\big(x_0+\phi^{-1}_{x_0}(t),t\big)|t\in[0,1-x_0]\}$ will be called the \textit{$x_0$-curve}  $\phi_{x_0}$, which is illustrated in Figure  \ref{fig:1}(a).  
	Denote the set of all curves $\phi$ as
	$$\Phi:=\{\phi: [0,1]\rightarrow [0,1]\ |\ \phi {\rm\ is\ a\ strictly\ increasing\  continuous\  function\   with\  } \phi(0)=0\ {\rm and}\ \phi(1)=1\},$$
	and let 
	$$\Phi_{x_0}:=\{\phi_{x_0}: [0,1-x_0]\rightarrow [0,1-x_0]\ |\ \phi_{x_0} {\rm\ is\ a\ strictly\  increasing\ continuous\ function\  with}\ $$
	$$\ \ \ \ \ \ \ \ \ \ \ \ \phi_{x_0}(0)=0\ {\rm and}\ \phi_{x_0}(1-x_0)=1-x_0\}$$
	be the set of all $x_0$-curves $\phi_{x_0}$. Observe that for any fixed $0\le x_0<1$ and any $\phi\in\Phi$, the function defined by  $\phi_{x_0}(t)=(1-x_0)\phi(\frac{t}{1-x_0})$, $t\in[0,1-x_0]$ belongs to $\Phi_{x_0}$. Conversely, for any $\phi_{x_0}\in\Phi_{x_0}$,
	$\frac{\phi_{x_0}\big((1-x_0)t\big)}{1-x_0}$, $t\in[0,1]$ is an element of $\Phi$. 

		We now present the notion of the track section of a quasi-copula along a curve $\phi\in\Phi$.
	
	\begin{defn}
		\label{track} 	For any curve $\phi\in\Phi$, the \textit{track section} (or \textit{curvilinear section}) of a quasi-copula  $Q$ along $\phi$ is the function $\Gamma_{Q,\phi}:$ $[0,1]\rightarrow [0,1]$ defined by $\Gamma_{Q,\phi}(t):=Q\big(t,\phi(t)\big)$, $t\in[0,1]$.  In the particular case where  $\phi={\rm id}$, i.e., $\phi(t)=t$ for all $t\in[0,1]$, the track section reduces to the \textit{diagonal section} of $Q$, denoted by  $\Gamma_{Q, {\rm id}}(t)=Q(t,t)=\delta_Q(t)$. 
	\end{defn}

Next, we proceed to introduce the definition of a  sub-track section along an $x_0$-curve $\phi_{x_0}\in\Phi_{x_0}$.
	
	\begin{defn}\label{sub-track}
		Given a quasi-copula $Q$ and $x_0\in[0, 1)$, the \textit{sub-track section} (or \textit{sub-curvilinear section}) of $Q$ along the $x_0$-curve $\phi_{x_0}$  is the function $\Gamma_{Q,\phi_{x_0}}:$ $[0,1-x_0]\rightarrow [0,1-x_0]$ defined by
		\begin{equation}\label{df:sub-track}
		\Gamma_{Q,\phi_{x_0}}(t):=Q\big(x_0+\phi^{-1}_{x_0}(t),t\big). 
		\end{equation}
	\end{defn}

The sub-track section of a copula $C$ along an $x_0$-curve $\phi_{x_0}$  admits a straightforward probabilistic interpretation. If $C$ is the joint distribution function of two uniformly distributed random variables $U$ and $V$ on $[0,1]$, then 
$$		
\Gamma_{C,\phi_{x_0}}(t)=C\big(x_0+\phi^{-1}_{x_0}(t), t\big)=\mathbb{P}\big(\max\big\{\phi_{x_0}(\max\{U-x_0,0\}), V\big\}\le t\big). 
$$

\begin{rem}
	\label{4} ($i$) In the special case where $x_0=0$, we have  $\Gamma_{Q,\phi_{0}}(t)=Q\big(\phi^{-1}_{0}(t),t\big)=\Gamma_{Q,\phi}\big(\phi^{-1}(t)\big)$, where $\phi=\phi_0$ and $\Gamma_{Q,\phi}(t)$ is the track section of $Q$ along the curve $\phi$ given in Definition \ref{track}.	
	
	($ii$) Note that for any $x_0\in[0, 1)$, when $\phi_{x_0}(t)=t$ for all $t\in[0,1-x_0]$, the sub-track section  $\Gamma_{Q,\phi_{x_0}}(t)=Q(x_0+t,t)$ reduces to the \textit{sub-diagonal section} $\delta_{Q, x_0}(t)$ of the quasi-copula $Q$ \citep{2008quasi-copulas}. 
	\end{rem}

\begin{rem}  Using the equation $v=\phi_{x_0}(u-x_0)$, we obtain the $x_0$-curve  $\{\big(x_0+\phi^{-1}_{x_0}(t), t\big)|t\in[0,1-x_0]\}$ by setting $v=t$. Then, in Definition \ref{sub-track}, the corresponding expression of sub-track section is given by  $Q\big(x_0+\phi^{-1}_{x_0}(t), t\big)$. This formulation differs from the definition of a track section of $Q$ along a curve $\phi\in \Phi$ as presented in Definition \ref{track}, where the expression is instead chosen as $u=t$, resulting in  the curve $\{\big(t, \phi(t)\big)|t\in[0,1]\}$ and the  track section  $Q\big(t, \phi(t)\big)$. In the special case where $x_0=0$, from Remark \ref{4}($i$), we know that the two definitions coincide essentially. Equivalently, the sub-track section of $Q$ along the $x_0$-curve $\phi_{x_0}$ may alternatively be defined as $Q\big(x_0+t, \phi_{x_0}(t)\big)$, $t\in[0,1-x_0]$,   which is essentially consistent with Definition \ref{sub-track}. The expression adopted in Definition \ref{sub-track} is motivated by its convenience for the analytical arguments in the subsequent sections. 
\end{rem}

\begin{example}\label{ex:1-}
For any $x_0$-curve $\phi_{x_0}\in \Phi_{x_0}$, the sub-track  of the product copula $\Pi$, the Fr\'{e}chet-Hoeffding lower bound  $W$ and the Fr\'{e}chet-Hoeffding upper bound  $M$ along this curve are given respectively by  $t\big(x_0+\phi_{x_0}^{-1}(t)\big)$, $\max\{x_0+\phi_{x_0}^{-1}(t)+t-1,0\}$  and $\min\{t, x_0+\phi_{x_0}^{-1}(t)\}$, $t\in[0,1-x_0]$. In particular, if $\phi_{x_0}(t)\le t$ for all $t\in[0,1-x_0]$, then $t\le x_0+\phi_{x_0}^{-1}(t)$ for every  $0\le x_0<1$. In this case, the sub-track section  of  the Fr\'{e}chet-Hoeffding upper bound $M$ along the $x_0$-curve $\phi_{x_0}$  is $\min\{t, x_0+\phi_{x_0}^{-1}(t)\}=t$, $t\in[0,1-x_0]$. 
\end{example}

In the following, we define a class of functions associated with the $x_0$-curve $\phi_{x_0}$.	

	\begin{defn}\label{sub-track1}
		Given $x_0\in[0, 1)$ and an $x_0$-curve $\phi_{x_0}\in\Phi_{x_0}$, a function $\Gamma_{\phi_{x_0}}$: $[0,1-x_0]\rightarrow$ $[0,1-x_0]$ is called a \textit{$\phi_{x_0}$-function}  if it fulfills   the following conditions:  		\begin{itemize}
			\item[(S1)] $ \max\{x_0+t+ \phi^{-1}_{x_0}(t)-1, 0\}\le \Gamma_{\phi_{x_0}}(t)\le \min\{x_0+\phi^{-1}_{x_0}(t), t\}$ for every $t\in[0,1-x_0]$; 
			\item[(S2)] $0\le \Gamma_{\phi_{x_0}}(t_{2})-\Gamma_{\phi_{x_0}}(t_{1})\le t_{2}-t_{1}+	\phi^{-1}_{x_0}(t_{2})-\phi^{-1}_{x_0}(t_{1})$ for any $0\le t_{1}\le t_{2}\le 1-x_0$.	 
		\end{itemize}
	\end{defn}

\begin{rem} For a $\phi_{x_0}$-function $\Gamma_{\phi_{x_0}}$, 
Condition (S1) implies that $\Gamma_{\phi_{x_0}}(0)=0$ and  $\Gamma_{\phi_{x_0}}(1-x_0)=1-x_0$. 
\end{rem}
The sub-track section $\Gamma_{Q,\phi_{x_0}}$ of any quasi-copula $Q$ along an  $x_0$-curve $\phi_{x_0}$ is necessarily a $\phi_{x_0}$-function. Indeed, the conditions stated above follow directly from the basic properties of quasi-copulas and copulas. Specifically, 
Condition (S1) is an  immediate consequence of the inequality $W(u,v)\le Q(u,v)\le M(u,v)$, which holds for all $(u,v)\in[0, 1]^2$. Condition (S2), in turn, follows from  the Lipschitz condition \eqref{Lp} together with the nondecreasing property of quasi-copulas. 

The following examples presents some typical  $\phi_{x_0}$-functions. 
\begin{example}\label{ex:1}
Let $\phi_{x_0}$ be an $x_0$-curve given by $\phi_{x_0}(t)=(1-x_0)\big(\frac{t}{1-x_0}\big)^{\beta}$, $t\in[0, 1-x_0]$, $\beta \ge 1$. It can be verified that the function $g(t)=t\big(x_0+(1-x_0)(\frac{t}{1-x_0})^{1/\beta}\big)^{\alpha}$, $0\le \alpha\le 1$ satisfies  Conditions  (S1) and (S2), and thus  is a $\phi_{x_0}$-function.
\end{example}

\begin{example}\label{ex:1+}
Consider an $x_0$-curve  $\phi_{x_0}(t)=\frac{(1-x_0)t}{2(1-x_0)-t}$, $t\in[0, 1-x_0]$. For any $\theta\in[-1,1]$, a direct verification shows that the function $g(t)=\frac{t\big((2-x_0)t+x_0(1-x_0)\big)}{t+(1-x_0)\big(1-\theta(1-x_0-t)(1-t)\big)}$,  $t\in[0,1-x_0]$  satisfies Conditions (S1) and (S2). Consequently, $g$ is a $\phi_{x_0}$-function.
\end{example}

\subsection{Notational preliminaries for sub-track sections}

For any fixed $x_0\in[0,1)$ and an $x_0$-curve $\phi_{x_0}\in\Phi_{x_0}$, we introduce the following sub-domains of the unit square, which are illustrated in Figure \ref{fig:1} and will be used in the subsequent analysis:
$$T_{U,\phi_{x_0}}:=\big\{(u,v)\in[0,1]^2|\phi_{x_0}(\max\{u-x_0,0\})\le v\big\},\ \ T_{L,\phi_{x_0}}:=\big\{(u,v)\in[0,1]^2|\phi_{x_0}(\max\{u-x_0,0\})> v\big\},$$
$$D_{U,x_0}:=[0,x_0]\times[1-x_0,1],\ \ D_{L,x_0}:=[x_0,1]\times[0,1-x_0],$$
$$S_{L,x_0}:=[0,x_0]^2,\ \ S_{U,x_0}:=[1-x_0,1]^2\ \ {\rm and}\ \  S_{x_0}:=S_{L,x_0}\cup S_{U,x_0}.$$

\begin{figure}
	\centering
	\begin{minipage}[t]{0.48\linewidth}
		\centerline{
			\includegraphics[scale=0.666]{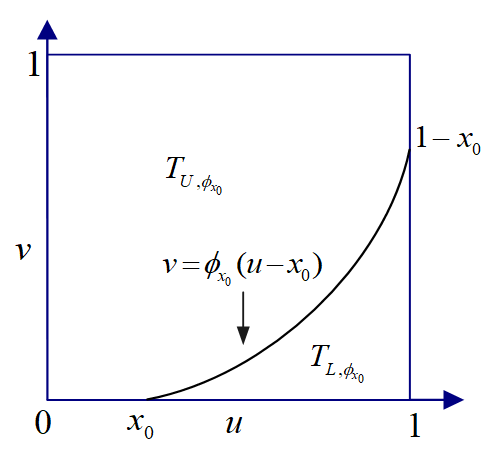}}
		\centerline{(a)}
	\end{minipage}%
	\hfill
	\begin{minipage}[t]{0.46\linewidth}
		\centerline{
			\includegraphics[scale=0.626]{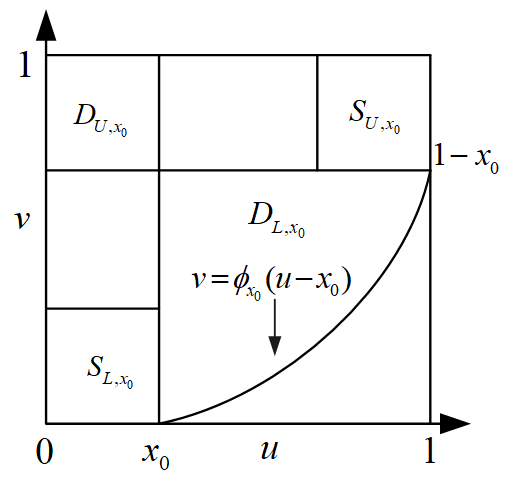}}
		\centerline{(b)}
	\end{minipage}
	\caption{ (a) The  $x_0$-curve $\phi_{x_0}$ as well as sub-domains  $T_{U,\phi_{x_0}}$ and $T_{L,\phi_{x_0}}$. (b) Sub-domains  $D_{U,x_0}$, $D_{L,x_0}$, $S_{U,x_0}$ and $S_{L,x_0}$.
		\label{fig:1}}
\end{figure}

Note that $T_{U,\phi_{x_0}}\cup T_{L,\phi_{x_0}}=[0,1]^2$. If $x_0< 1/2$, then the  sub-domains $S_{U,x_0}$ and $S_{L,x_0}$ are disjoint. By contrast, the  sub-domains $D_{U,x_0}$ and $D_{L,x_0}$ intersect in exactly the single point $(x_0, 1-x_0)$, regardless of the value of $x_0$.

We further introduce two auxiliary functions $m_{\phi_{x_0}}:$
$ [0, 1]^2
\rightarrow [0, 1-x_0]$ and $M_{\phi_{x_0}}:$
$ [0, 1]^2
\rightarrow [0, 1-x_0]$ defined as follows:
\begin{equation}\label{mx}
m_{\phi_{x_0}}(u,v):=\min\big\{\phi_{x_0}(\max\{u-x_0, 0\}),v\big\}
\end{equation}
and 
\begin{equation}\label{Mx}
M_{\phi_{x_0}}(u,v):=\min\Big\{\max\big\{\phi_{x_0}(\max\{u-x_0, 0\}),v\big\},1-x_0\Big\}.
\end{equation}
 The truncation at 0 and $1-x_0$, respectively, ensures 
that $\Gamma_{\phi_{x_0}}$ and any function derived from it can be applied to $m_{\phi_{x_0}}$ and $M_{\phi_{x_0}}$, respectively, for all   $(u, v)\in[0, 1]^2$. It is easy to see that 
$$m_{\phi_{x_0}}(u,v)=\min\{\phi_{x_0}(u-x_0),v\}\ \ {\rm and}\ \ M_{\phi_{x_0}}(u,v)=\max\{\phi_{x_0}(u-x_0),v\}$$
for all $(u,v)\in D_{L,x_0}$. Note also that $(u, v)\in T_{U,\phi_{x_0}}\cap D_{L,x_0}$ if and only if $m_{\phi_{x_0}}(u,v)= \phi_{x_0}(u-x_0)$ and $M_{\phi_{x_0}}(u,v)=v$, whereas $(u, v)\in T_{L,\phi_{x_0}}\cap D_{L,x_0}$ if and only if $m_{\phi_{x_0}}(u,v)= v$ and $M_{\phi_{x_0}}(u,v)=\phi_{x_0}(u-x_0)$.

For an arbitrary $\phi_{x_0}$-function  $\Gamma_{\phi_{x_0}}$,  we consider the functions $\hat \Gamma^I_{\phi_{x_0}}:[0,1-x_0]\rightarrow [0,1-x_0]$ and $\hat \Gamma^{II}_{\phi_{x_0}}:[0,1-x_0]\rightarrow [0,1-x_0]$ 
defined by
$$\hat \Gamma^I_{\phi_{x_0}}(t):=\phi^{-1}_{x_0}(t)-\Gamma_{\phi_{x_0}}(t)\ \ 
{\rm and} 
\ \  \hat \Gamma^{II}_{\phi_{x_0}}(t):=t-\Gamma_{\phi_{x_0}}(t),\ t\in[0,1-x_0].$$

It is immediate to verify that $\hat \Gamma^I_{\phi_{x_0}}$ and $\hat\Gamma^{II}_{\phi_{x_0}}$
satisfy the following properties, all of which follow directly from Conditions (S1) and (S2):
\begin{itemize}
	\item[(D1)]   $\hat\Gamma^I_{\phi_{x_0}}(0)=\hat\Gamma^{II}_{\phi_{x_0}}(0)=\hat\Gamma^I_{\phi_{x_0}}(1-x_0)=\hat\Gamma^{II}_{\phi_{x_0}}(1-x_0)=0$;
	\item[(D2)] $t_1-t_2\le \hat\Gamma^{I}_{\phi_{x_0}}(t_2)-\hat\Gamma^{I}_{\phi_{x_0}}(t_1) \le \phi^{-1}_{x_0}(t_{2})-\phi^{-1}_{x_0}(t_{1})$ and $\phi^{-1}_{x_0}(t_{1})-\phi^{-1}_{x_0}(t_{2})\le \hat\Gamma^{II}_{\phi_{x_0}}(t_2)-\hat \Gamma^{II}_{\phi_{x_0}}(t_1) \le t_2-t_1$  for all $0\le t_1\le t_2\le 1-x_0$; 
	\item[(D3)]  $\max\{\phi^{-1}_{x_0}(t)-t,-x_0\}\le \hat\Gamma^{I}_{\phi_{x_0}}(t)\le \min\{\phi^{-1}_{x_0}(t),1-x_0-t\}$ and $\max\{0,t-x_0-\phi^{-1}_{x_0}(t)\}\le \hat\Gamma^{II}_{\phi_{x_0}}(t)\le \min\{t,1-x_0-\phi^{-1}_{x_0}(t)\}$ for every $0\le t\le 1-x_0$.	 
	\item[(D4)] Let $U$  and $V$ be random variables uniformly distributed on  [0, 1], and  $C$ be the copula of the pair $(U,V)$. Suppose that  $\Gamma_{\phi_{x_0}}(t)$,
	$t\in[0, 1-x_0]$ is the sub-track section of  $C$ along the $x_0$-curve $\phi_{x_0}$. Then the distribution functions of $\max\big\{\phi_{x_0}(\max\{U-x_0,0\}), V\big\}$ and $\min\big\{\phi_{x_0}(\max\{U-x_0,0\}), V\big\}$ are given by  $t-\hat\Gamma^{II}_{\phi_{x_0}}(t)$ and $t+x_0 +\hat\Gamma^{I}_{\phi_{x_0}}(t)$, $t\in[0, 1-x_0]$ respectively. Consequently, 
	\begin{equation}
	x_0+\hat\Gamma^{I}_{\phi_{x_0}}(t)+\hat\Gamma^{II}_{\phi_{x_0}}(t)=\mathbb{P}\big(\min\big\{\phi_{x_0}(\max\{U-x_0,0\}), V\big\}\le t<\max\big\{\phi_{x_0}(\max\{U-x_0,0\}), V\big\}\big).
	\end{equation}
\end{itemize}

The following auxiliary functions will also be required:
$$q^I_{\phi_{x_0}}(u,v):=\max\big\{\hat\Gamma^{I}_{\phi_{x_0}}(t)|t\in[m_{\phi_{x_0}}(u,v), M_{\phi_{x_0}}(u,v)] \big\},\ (u,v)\in[0,1]^2$$
and 
$$ q^{II}_{\phi_{x_0}}(u,v):=\max\big\{\hat\Gamma^{II}_{\phi_{x_0}}(t)|t\in[m_{\phi_{x_0}}(u,v), M_{\phi_{x_0}}(u,v)] \big\},\ (u,v)\in[0,1]^2.$$

\section{Constructions of  quasi-copulas and copulas with a given sub-track section}

For a $\phi_{x_0}$-function $\Gamma_{\phi_{x_0}}$, denote by $\mathbb{Q}_{\Gamma_{\phi_{x_0}}}$ (respectively, $\mathbb{C}_{\Gamma_{\phi_{x_0}}}$)  the class of all quasi-copulas (respectively, copulas) whose sub-track along the $x_0$-curve $\phi_{x_0}$ is $\Gamma_{\phi_{x_0}}$. For the case $x_0 = 0$,   construction methods and best-possible bounds for quasi-copulas in $\mathbb{Q}_{\Gamma_{\phi_{x_0}}}$ are already known \citep{De2019On,2022Best}. Consequently, in this paper, we restrict our attention to the case  $x_0\in(0,1)$. One possible approach for  constructing quasi-copulas with a given sub-track section is to apply bilinear transformations to quasi-copulas with prescribed track sections. From a geometrical perspective, the idea consists in rescaling and shifting a given quasi-copula such 
that it is supported on the sub-domain $D_{L,x_0}$ rather than on the whole $[0,1]^2$, and then filling the gaps on $[0, 1]^2\setminus D_{L,x_0}$ in an appropriate manner. As a result of this rescaling and  transformation, the value of the constructed  quasi-copula at the point $(x_0, 1-x_0)$
must be 0. In accordance with the results of \cite{2004W}, such a quasi-copula $Q$ admits a representation as a $W$-ordinal sum, that is,
there exist quasi-copulas $Q_1$ and $Q_2$ such that
\begin{equation}
\label{eq:W}Q(u,v)
=\left\{\begin{array}{ll}
x_0 Q_1\big(\frac{u}{x_0},\frac{x_0+v-1}{x_0}\big),& \ {\rm if}\ \ (u,v)\in D_{U,x_0},\\
(1-x_0) Q_2\big(\frac{u-x_0}{1-x_0},\frac{v}{1-x_0}\big),&  \ {\rm if}\ \ (u,v)\in D_{L,x_0},\\
W(u,v),&  \ {\rm otherwise}.
\end{array}\right.
\end{equation}

Notice that $Q$ is indeed a quasi-copula. It will be denoted by $(\langle 0,x_0, Q_1 \rangle, \langle x_0, 1, Q_2 \rangle)^W$. Furthermore, if both $Q_1$ and $Q_2$ are copulas, then the resulting function $Q$ is also  a copula \citep{2004W}.

\begin{thm}\label{prop:G}
	For a fixed $x_0\in (0, 1)$, let $\phi_{x_0}\in\Phi_{x_0}$ be an $x_0$-curve satisfying $\phi_{x_0}(t)\le t$ for all $t\in[0,1-x_0]$, and let $\Gamma_{\phi_{x_0}}$ 
	be a $\phi_{x_0}$-function. Take an arbitrary quasi-copula $Q_1$, and let $Q_2$ be a quasi-copula
whose track section along the curve $\phi(t)=\frac{\phi_{x_0}\big(t(1-x_0)\big)}{1-x_0}$, $t\in[0,1]$ is given by  
	\begin{equation}
	\label{eq:G} \Gamma_{\phi}(t)
=\frac{	\Gamma_{\phi_{x_0}}\Big(\phi_{x_0}\big(t(1-x_0)\big)\Big)}{1-x_0},\ t\in[0,1]. 
	\end{equation}
	Then, the $W$-ordinal sum $Q=(\langle 0,x_0, Q_1 \rangle, \langle x_0, 1, Q_2 \rangle)^W$ defined by \eqref{eq:W} is a quasi-copula satisfying  $\Gamma_{Q,\phi_{x_0}}=\Gamma_{\phi_{x_0}}$, i.e., $Q \in \mathbb{Q}_{\Gamma_{\phi_{x_0}}}$. Consequently, for every $\phi_{x_0}$-function  $\Gamma_{\phi_{x_0}}$  with $\phi_{x_0}(t)\le t$, $t\in [0,1-x_0]$, there exists a quasi-copula $Q$ 
	belonging to $\mathbb{Q}_{\Gamma_{\phi_{x_0}}}$. 
\end{thm}

\begin{proof}
	It suffices to verify that, for every fixed  $x_0\in(0,1)$, the function $\Gamma_{\phi}$ defined by 	\eqref{eq:G} is indeed a track section, and that $Q\big( x_0+\phi^{-1}_{x_0}(t), t \big)= \Gamma_{\phi_{x_0}}(t)$ for all
	$t\in[0, 1-x_0]$. Since $\Gamma_{\phi_{x_0}}(0)=0$ and	$\Gamma_{\phi_{x_0}}(1-x_0)=1-x_0$, we obtain directly  that
	$\Gamma_{\phi}(0)=0$ and	$\Gamma_{\phi}(1)=1$. From Condition (S1), it also follows that   
	$$
	\Gamma_{\phi}(t)=\frac{	\Gamma_{\phi_{x_0}}\Big(\phi_{x_0}\big(t(1-x_0)\big)\Big)}{1-x_0}\ge \max\{0,\frac{t(1-x_0)+ x_0+\phi_{x_0}\big(t(1-x_0)\big)-1}{1-x_0}\}=\max\{0,t+\phi(t)-1\}
	$$
and 
$$\Gamma_{\phi}(t)=\frac{	\Gamma_{\phi_{x_0}}\Big(\phi_{x_0}\big(t(1-x_0)\big)\Big)}{1-x_0}\le\min\{\frac{\phi_{x_0}\big(t(1-x_0)\big)}{1-x_0},\  t+\frac{x_0}{1-x_0}\}=\min\{\phi(t),\  t+\frac{x_0}{1-x_0}\}\le \phi(t). $$	
From the assumption that $\phi_{x_0}(t)\le t$, $t\in[0, 1-x_0]$, we know that $\phi(t)=\frac{\phi_{x_0}\big(t(1-x_0)\big)}{1-x_0}\le t$, $t\in[0,1]$ such that  $\Gamma_{\phi}(t)\le\min\{t,\phi(t)\}$ for all $t\in[0, 1]$. Moreover, because $\Gamma_{\phi_{x_0}}$ fulfills Condition (S2), it follows immediately that, for all $t_1$, $t_2\in[0, 1]$ with $t_1 \le t_2$,
$$0\le \Gamma_{\phi}(t_2)-\Gamma_{\phi}(t_1)=\frac{\Gamma_{\phi_{x_0}}\Big(\phi_{x_0}\big(t_2(1-x_0)\big)\Big)-\Gamma_{\phi_{x_0}}\Big(\phi_{x_0}\big(t_1(1-x_0)\big)\Big)}{1-x_0}$$
$$\ \ \ \ \ \ \ \ \ \ \ \ \ \ \ \ \ \ \ \ \   =\frac{\phi_{x_0}\big(t_2(1-x_0)\big)-\phi_{x_0}\big(t_1(1-x_0)\big)+(t_2-t_1)(1-x_0)}{1-x_0}=t_2-t_1+\phi(t_2)-\phi(t_1).$$
Hence, from \cite{De2019On}, we know that  $\Gamma_{\phi}$ is a track section along the curve $\phi$. 

Now assume that $t\in[0, 1-x_0]$. Then $\big( x_0+\phi^{-1}_{x_0}(t), t\big)\in D_{L,x_0}$ and 
$$Q\big(x_0+\phi^{-1}_{x_0}(t), t\big)=(1-x_0)Q_2\big(\frac{\phi^{-1}_{x_0}(t)}{1-x_0},\frac{t}{1-x_0}\big)=(1-x_0)Q_2\Big(\frac{\phi^{-1}_{x_0}(t)}{1-x_0}, \frac{\phi_{x_0}\big((\frac{\phi^{-1}_{x_0}(t)}{1-x_0})(1-x_0)\big)}{1-x_0}\Big)$$
$$\ \ \ \ \ \ \ \ \ \ \ \ \ \ \ \ \   =(1-x_0)Q_2\Big(\frac{\phi^{-1}_{x_0}(t)}{1-x_0}, \phi\big(\frac{\phi^{-1}_{x_0}(t)}{1-x_0}\big) \Big)=(1-x_0)\Gamma_{\phi}\big(\frac{\phi^{-1}_{x_0}(t)}{1-x_0}\big)=\Gamma_{\phi_{x_0}}(t),$$
which implies that $Q\in\mathbb{Q}_{\Gamma_{\phi_{x_0}}}$. 
\end{proof}

We now illustrate the results of Theorem \ref{prop:G} through several examples. Specifically, the sub-track sections are taken to be those of the largest
and smallest quasi-copulas,  respectively. The resulting $W$-ordinal sums $Q$ with these sub-track sections are then  presented. 

\begin{example}\label{ex:w}
	Given $x_0\in(0,1)$, consider the $x_0$-curve $\phi_{x_0}(t)=(1-x_0)\big(\frac{t}{1-x_0}\big)^{\beta}$, $t\in[0, 1-x_0]$, $\beta \ge 1$ and the smallest sub-track section  $\Gamma_{W,\phi_{x_0}}(t)=\max\{x_0+t+t^{1/\beta}(1-x_0)^{1-1/\beta}-1,0\}$, $t\in[0,1-x_0]$. Then,  $\Gamma_{\phi}$ of type \eqref{eq:G} is given by $\Gamma_{\phi}(t)=\max\{t+t^{\beta}-1,0\}$, $t\in[0,1]$, which is exactly the track section of the Fr\'{e}chet-Hoeffding lower bound $W$ along the curve $\phi(t)=t^{\beta}$, $t\in[0,1]$. Setting $Q_1=W$ and $Q_2=B_{\Gamma_{\phi}}$ of type  \eqref{eq:Bertino copula+in}, a straightforward calculation yields that, for $(u,v)\in D_{L,x_0}$, 
	\begin{eqnarray*}
	&&Q(u,v)=  (1-x_0)Q_2\big(\frac{u-x_0}{1-x_0},\frac{v}{1-x_0}\big) \\
	&=&	\left\{\begin{array}{ll}
		v-(1-x_0)\min\big\{t^{\beta}-\max\{t+t^{\beta}-1,0\}|t\in[(\frac{v}{1-x_0})^{1/\beta},\frac{u-x_0}{1-x_0}]\big\},& \ \frac{v}{1-x_0}<\big(\frac{u-x_0}{1-x_0}\big)^{\beta},\\
		u-x_0-(1-x_0)\min\big\{t-\max\{t+t^{\beta}-1,0\}|t\in[\frac{u-x_0}{1-x_0},(\frac{v}{1-x_0})^{1/\beta}]\big\},&  \ \frac{v}{1-x_0}\ge\big(\frac{u-x_0}{1-x_0}\big)^{\beta}
	\end{array}\right. \nonumber \\
	&=&\left\{\begin{array}{ll}
		v-(1-x_0)\min\{1-t,t^{\beta}|t\in[(\frac{v}{1-x_0})^{1/\beta},\frac{u-x_0}{1-x_0}] \},& \ \frac{v}{1-x_0}<\big(\frac{u-x_0}{1-x_0}\big)^{\beta},\\
		u-x_0-(1-x_0)\min\{1-t^{\beta}, t|t\in[\frac{u-x_0}{1-x_0}, (\frac{v}{1-x_0})^{1/\beta}] \},&  \ \frac{v}{1-x_0}\ge \big(\frac{u-x_0}{1-x_0}\big)^{\beta} 
	\end{array}\right.\nonumber\\
	&=&\left\{\begin{array}{ll}
	v-\min\{v,1-x_0-(u-x_0) \},& \ \frac{v}{1-x_0}<\big(\frac{u-x_0}{1-x_0}\big)^{\beta},\\
	u-x_0-\min\{u-x_0,1-x_0-v \},&  \ \frac{v}{1-x_0}\ge \big(\frac{u-x_0}{1-x_0}\big)^{\beta}
\end{array}\right.\nonumber\\
	&=&\max\{0,u+v-1\}=W(u,v),
\end{eqnarray*}
	and, for $(u,v)\in D_{U,x_0}$,  $$Q(u,v)=x_0 Q_1\Big(\frac{u}{x_0},\frac{x_0+v-1}{x_0}\Big)=x_0\max\Big\{\frac{u}{x_0}+\frac{x_0+v-1}{x_0}-1,0\Big\}=W(u,v).$$
Hence, the corresponding $W$-ordinal sum $Q$ of type  \eqref{eq:W} is exactly the Fr\'{e}chet-Hoeffding  lower bound  $W$,  and consequently belongs to  $\mathbb{Q}_{\Gamma_{W, \phi_{x_0}}}$. 
\end{example}

\begin{example}\label{ex:m}
	Given $x_0\in(0,1)$, consider the $x_0$-curve $\phi_{x_0}(t)=(1-x_0)\big(\frac{t}{1-x_0}\big)^{\beta}$, $t\in[0, 1-x_0]$, $\beta \ge 1$ and the largest sub-track section  $\Gamma_{M,\phi_{x_0}}(t)=\min\{x_0+\phi^{-1}_{x_0}(t),t\}=t$, $t\in[0,1-x_0]$. Then, $\Gamma_{\phi}$ of type \eqref{eq:G} is given by $\Gamma_{\phi}(t)=\min\{\frac{x_0}{1-x_0}+t,t^{\beta}\}=t^{\beta}=\min\{t,t^{\beta}\}$, $t\in[0,1]$,   which is exactly the track section of the the Fr\'{e}chet-Hoeffding upper bound $M$ along the curve $\phi(t)=t^{\beta}$, $t\in[0,1]$ and $\beta\ge 1$. Choose $Q_1=M$ and $Q_2=A_{\Gamma_{\phi}}$ as defined in  \eqref{eq:A copula}. In this case, for $(u,v)\in D_{U,x_0}$,  $x_0 Q_1\big(\frac{u}{x_0},\frac{x_0+v-1}{x_0}\big)=\min\{u, x_0+v-1\}$, and for $(u,v)\in D_{L,x_0}$,  
	\begin{equation*}(1-x_0) Q_2\big(\frac{u-x_0}{1-x_0},\frac{v}{1-x_0}\big)
		=\left\{\begin{array}{ll}
			\min\Big\{v,u-x_0 -(1-x_0)\max\{t-t^{\beta}&  \\ \ \ \ \ \ \ \ \ \ \ \ \ \ \ \ \ \ \ \ \ \ \ \ \ |t\in[(\frac{v}{1-x_0})^{1/\beta},\frac{u-x_0}{1-x_0}]\}\Big\}, & {\rm if}\  (u,v)\in D_{L,x_0}\ \ {\rm and}\ \ v<\frac{(u-x_0)^{\beta}}{(1-x_0)^{\beta-1}},  \\
			\min\{u-x_0, 	v\},&  {\rm if}\  (u,v)\in D_{L,x_0}\ \ {\rm and}\ \ v\ge \frac{(u-x_0)^{\beta}}{(1-x_0)^{\beta-1}}. 
		\end{array}\right.
	\end{equation*}
Hence, the $W$-ordinal sum $Q$ constructed via  \eqref{eq:W} is expressed as 
		\begin{equation*} Q(u,v)	=\left\{\begin{array}{ll}
			\min\{u, x_0+v-1\} & {\rm if}\ (u,v)\in D_{U,x_0},\\		
		\min\Big\{v,u-x_0 -(1-x_0)\max\{t-t^{\beta}|t\in[(\frac{v}{1-x_0})^{1/\beta},\frac{u-x_0}{1-x_0}]\}\Big\}, & {\rm if}\  (u,v)\in D_{L,x_0}\ \ {\rm and}\ \ v<\frac{(u-x_0)^{\beta}}{(1-x_0)^{\beta-1}},  \\
			\min\{u-x_0, 	v\},&  {\rm if}\  (u,v)\in D_{L,x_0}\ \ {\rm and}\ \ v\ge \frac{(u-x_0)^{\beta}}{(1-x_0)^{\beta-1}},\\
			W(u,v), &  {\rm otherwise},
		\end{array}\right.
	\end{equation*}
which is not equal to the Fr\'{e}chet-Hoeffding upper bound  $M$.
\end{example}

Based on the result of Theorem \ref{prop:G}, we show that the class  $\mathbb{C}_{\Gamma_{\phi_{x_0}}}$ is also nonempty and  present a characterization of the sub-track section  of a copula  along an $x_0$-curve $\phi_{x_0}\in\Phi_{x_0}$.

\begin{prop}
\label{cor:1}	
	Let $\phi_{x_0}$ be an $x_0$-curve in $\Phi_{x_0}$  with $\phi_{x_0}(t)\le t$, $t\in [0,1-x_0]$.

(i) For every $\phi_{x_0}$-function  $\Gamma_{\phi_{x_0}}$,  there exists a copula $C$ 
belonging to $\mathbb{C}_{\Gamma_{\phi_{x_0}}}$.

(ii) A function $g:[0,1-x_0]\rightarrow [0,1-x_0]$ is a sub-track section $\Gamma_{C, \phi_{x_0}}$ of a copula $C$ along the $x_0$-curve $\phi_{x_0}$ if and only if $g$ is a $\phi_{x_0}$-function. 
\end{prop}

\begin{proof} ($i$) Let $C_1$ be an arbitrary copula and $C_2$ be  a copula with the track section $\Gamma_{\phi}$ given in \eqref{eq:G}. Consider the $W$-ordinal sum $C=(\langle 0,x_0, C_1 \rangle, \langle x_0, 1, C_2 \rangle)^W$ defined by \eqref{eq:W}. By an argument analogous to that used in the proof of Theorem  \ref{prop:G}, it follows that $C$ is a copula and $\Gamma_{C, \phi_{x_0}}=\Gamma_{\phi_{x_0}}$, that is, $C\in \mathbb{C}_{\Gamma_{\phi_{x_0}}}$.

($ii$)	If $g$ is the sub-track section $\Gamma_{C, \phi_{x_0}}$ of a copula $C$ along the $x_0$-curve $\phi_{x_0}$, then $g$ is obviously a $\phi_{x_0}$-function. 
		Conversely, if $g$ is a $\phi_{x_0}$-function, it follows from the statement ($i$)  that the $W$-ordinal sum $C=(\langle 0,x_0, C_1 \rangle, \langle x_0, 1, C_2 \rangle)^W$ defined by \eqref{eq:W} is a copula satisfying $\Gamma_{C,\phi_{x_0}}=g$, where $C_1$ is an arbitrary copula and 
	$C_2$ is a copula with the given track section $\Gamma_{\phi}(t)=g\Big(\phi_{x_0}\big(t(1-x_0)\big)\Big)/(1-x_0)$, $t\in[0,1]$.
\end{proof}

\begin{example}\label{ex2}
	($i$) The $\phi_{x_0}$-function  $g(t)=t\big(x_0+(1-x_0)(\frac{t}{1-x_0})^{1/\beta}\big)^{\alpha}$, $0\le \alpha\le 1$, $t\in[0, 1-x_0]$, $\beta \ge 1$ presented in Example \ref{ex:1}, is the sub-track section of the Cuadras-Aug\'{e} family of copulas \citep{1981A} $C(u,v)=\min\{u,v\} (\max\{u,v\})^{\alpha}$, $(u,v)\in [0,1]^2$  along the $x_0$-curve $\phi_{x_0}(t)=(1-x_0)\big(\frac{t}{1-x_0}\big)^{\beta}$. In particular, when  $\alpha=0$, the Cuadras-Aug\'{e} copula reduces to the Fr\'{e}chet-Hoeffding upper bound $M(u,v)=\min\{u,v\}$, and its sub-track section is $g(t)=t$. When  $\alpha=1$, the Cuadras-Aug\'{e} copula reduces to the product copula $\Pi(u,v)=uv$, and its sub-track section along the $x_0$-curve $\phi_{x_0}(t)=(1-x_0)\big(\frac{t}{1-x_0}\big)^{\beta}$ is $g(t)=t\big(x_0+(1-x_0)(\frac{t}{1-x_0})^{1/\beta}\big)$.

\noindent ($ii$) The $\phi_{x_0}$-function  $g(t)=\frac{t\big((2-x_0)t+x_0(1-x_0)\big)}{t+(1-x_0)\big(1-\theta(1-x_0-t)(1-t)\big)}$, $\theta\in[-1,1]$, $t\in[0, 1-x_0]$ presented in Example \ref{ex:1+}  is the sub-track section of the Ali-Mikhail-Haq family of copulas \citep{Nelsen(2006)} $C(u,v)=uv/(1-\theta(1-u)(1-v))$, $\theta\in[-1,1]$, $(u,v)\in [0,1]^2$  along the $x_0$-curve $\phi_{x_0}(t)=\frac{(1-x_0)t}{2(1-x_0)-t}$. 
\end{example}

\section{Families of quasi-copulas and copulas with a prescribed  sub-track section}

In this section,  based on the $W$-ordinal sum construction established in Theorem \ref{prop:G} and Proposition \ref{cor:1},	 and employing the (quasi-)copulas $B_{\Gamma_{\phi}}$, $K_{\Gamma_{\phi}}$ and $A_{\Gamma_{\phi}}$ presented  in \eqref{eq:Bertino copula+in}-\eqref{eq:K copula}, we construct several families of quasi-copulas and copulas with a prescribed  sub-track section $\Gamma_{\phi_{x_0}}$. 

It should be noted that the  $W$-ordinal sum $Q=(\langle 0,x_0, Q_1 \rangle, \langle x_0, 1, Q_2 \rangle)^W$ defined by \eqref{eq:W} does not necessarily yield
symmetric quasi-copulas. Hence, we further construct several families of symmetric quasi-copulas with the same sub-track section. 

As a preliminary step, we recall a symmetrization procedure that obtains a symmetric quasi-copula from a non-symmetric (quasi-)copula. 

\begin{lem}\label{lem:sym}\citep{2008quasi-copulas} Given any (quasi-)copula $Q$, the function $Q^*:$ $[0, 1]^2$ $\rightarrow[0, 1]$ defined by $Q^*(u,v):=Q(\max\{u,v\},$ 
	$ \min\{u,v\})$ is a symmetric quasi-copula.
\end{lem}

Using the copula $B_{\Gamma_{\phi}}$ of type  \eqref{eq:Bertino copula+in}, we construct the copula $B_{\Gamma_{\phi_{x_0}}}$ with the prescribed sub-track section $\Gamma_{\phi_{x_0}}$, as well as the symmetric quasi-copula $B^*_{\Gamma_{\phi_{x_0}}}$ sharing the same  sub-track section $\Gamma_{\phi_{x_0}}$.

\begin{prop}\label{prop:Bertino copula*}
For a fixed $x_0\in(0,1)$, let  $\Gamma_{\phi_{x_0}}$ be a $\phi_{x_0}$-function, where   $\phi_{x_0}\in\Phi_{x_0}$ satisfies  $\phi_{x_0}(t)\le t$ for all $t\in[0,1-x_0]$. Then, the function  $B_{\Gamma_{\phi_{x_0}}}:$ $[0, 1]^2\rightarrow$ $[0, 1]$ defined by 
\begin{equation}\label{eq:Bx}
	B_{\Gamma_{\phi_{x_0}}}(u,v)
	=\left\{\begin{array}{ll}
		v-\min\{\hat\Gamma^{II}_{\phi_{x_0}}(t)|t\in[v, \phi_{x_0}(u-x_0)]\}, & \ {\rm if}\  (u,v)\in D_{L, x_0}\cap T_{L, \phi_{x_0}},\\
		u-x_0-\min\{\hat\Gamma^{I}_{\phi_{x_0}}(t)|t\in[\phi_{x_0}(u-x_0),v] \},&  \ {\rm if}\  (u,v)\in D_{L, x_0}\cap T_{U, \phi_{x_0}}, \\
		W(u,v),&  \ {\rm otherwise}, 
	\end{array}\right.
\end{equation}
is a copula belonging to  $\mathbb{C}_{\Gamma_{\phi_{x_0}}}$. In particular, the function $B^*_{\Gamma_{\phi_{x_0}}}:$ $[0, 1]^2\rightarrow$ $[0, 1]$ defined by
$$ B^*_{\Gamma_{\phi_{x_0}}}(u,v)\ \ \ \ \ \ \ \ \ \ \ \ \ \ \ \ \ \ \ \ \ \ \ \ \ \ \ \ \ \ \ \ \ \ \ \ \ \ \ \ \ \ \ \ \ \ \ \  \ \ \ \ \ \ \ \ \ \ \ \ \ \ \ \ \ \ \ \ \ \ \ \ \ \ \ \ \ \ \ \ \ \ \ \ \ \ \ \ \ \ \ \ \ \ \ \   \ \ \ \ \ \ \ \ \ \ \ \ \ \ \ \ \ \ \ \ \ \ \ \ \ \ \ \ \ \ \ \ \ \ \ \ \ \ \ \ \ \ \ \ \ \  \ \ \ \ \ \  $$
	\begin{equation}\label{eq:Bertino copula*}
=\left\{\begin{array}{ll}
		\min\{u,v\}-\min\big\{\hat\Gamma^{II}_{\phi_{x_0}}(t)|t\in[\min\{u,v\}, \phi_{x_0}(\max\{u,v\}-x_0)] \big\}, & \ {\rm if}\  (u,v)\in [0,1]^2 \backslash S_{x_0}\ \ {\rm and}\\  &  \min\{u,v\}<\phi_{x_0}\big(\max\{u,v\}-x_0\big),\\
		\max\{u,v\}-x_0-\min\big\{\hat\Gamma^{I}_{\phi_{x_0}}(t)|t\in[ \phi_{x_0}(\max\{u,v\}-x_0), \min\{u,v\}] \big\},&  \ {\rm if}\  (u,v)\in [0,1]^2 \backslash S_{x_0}\ \ {\rm and}\\
		& \min\{u,v\}\ge \phi_{x_0}\big(\max\{u,v\}-x_0\big), \\
		W(u,v),&  \ {\rm otherwise}, 
	\end{array}\right.
\end{equation}	
is a symmetric quasi-copula belonging to  $\mathbb{Q}_{\Gamma_{\phi_{x_0}}}$.
\end{prop}

\begin{proof}
	We shall show that $B_{\Gamma_{\phi_{x_0}}}$ is a $W$-ordinal sum 
	of the form $\big(\langle 0,x_0, W \rangle, \langle x_0, 1, B_{\Gamma_{\phi}} \rangle\big)^W$ 
	with $\Gamma_{\phi}$ defined by \eqref{eq:G}, and that  $B^*_{\Gamma_{\phi_{x_0}}}$ 
	is its symmetrization. Recall that the function $B_{\Gamma_{\phi}}$ of type  \eqref{eq:Bertino copula+in} is a copula with the track section $\Gamma_{\phi}$ along the curve $\phi$, and the $W$-ordinal sum $\big(\langle 0,x_0, W \rangle, \langle x_0, 1, B_{\Gamma_{\phi}} \rangle\big)^W$ is given by $$\big(\langle 0,x_0, W(u,v) \rangle, \langle x_0, 1, B_{\Gamma_{\phi}}(u,v) \rangle\big)^W
	=\left\{\begin{array}{ll}
(1-x_0)B_{\Gamma_{\phi}}\big(\frac{u-x_0}{1-x_0},\frac{v}{1-x_0}\big),& \ {\rm if}\  (u,v)\in D_{L,x_0},\\
W(u,v),&  \ {\rm otherwise}. 
	\end{array}\right.$$
By straightforward calculations, we obtain that for all $(u, v)\in D_{L,x_0}$, 
	\begin{eqnarray*}
	&& (1-x_0)B_{\Gamma_{\phi}}\big(\frac{u-x_0}{1-x_0},\frac{v}{1-x_0}\big) \\
	&=&	\left\{\begin{array}{ll}
		v-(1-x_0)\min\{\phi(t)-\Gamma_{\phi}(t)|t\in[\phi^{-1}(\frac{v}{1-x_0}),\frac{u-x_0}{1-x_0}]\},& \ \frac{v}{1-x_0}<\phi\big(\frac{u-x_0}{1-x_0}\big),\\
		u-x_0-(1-x_0)\min\{t-\Gamma_{\phi}(t)|t\in[\frac{u-x_0}{1-x_0},\phi^{-1}(\frac{v}{1-x_0})]\},&  \ \frac{v}{1-x_0}\ge\phi\big(\frac{u-x_0}{1-x_0}\big)
	\end{array}\right. \nonumber \\
	&=&\left\{\begin{array}{ll}
		v-\min\{\phi_{x_0}\big(t(1-x_0)\big)-\Gamma_{\phi_{x_0}}\big(\phi_{x_0}(t(1-x_0))\big)|t\in[\phi^{-1}(\frac{v}{1-x_0}),\frac{u-x_0}{1-x_0}] \},& \ v<\phi_{x_0}\big(u-x_0\big),\\
		u-x_0-\min\{t(1-x_0)-\Gamma_{\phi_{x_0}}\big(\phi_{x_0}(t(1-x_0))\big)|t\in[\frac{u-x_0}{1-x_0},\phi^{-1}(\frac{v}{1-x_0})] \},&  \ v\ge\phi_{x_0}\big(u-x_0\big)
	\end{array}\right.\nonumber\\
	&=&\left\{\begin{array}{ll}
		v-\min\{\phi_{x_0}(t)-\Gamma_{\phi_{x_0}}\big(\phi_{x_0}(t)\big)|t\in[\phi_{x_0}^{-1}(v), u-x_0]\},& \ v<\phi_{x_0}\big(u-x_0\big),\\
		u-x_0-\min\{t-\Gamma_{\phi_{x_0}}\big(\phi_{x_0}(t)\big)|t\in[u-x_0,\phi_{x_0}^{-1}(v)]\},&  \ v\ge\phi_{x_0}\big(u-x_0\big)
	\end{array}\right.\nonumber\\
	&=&\left\{\begin{array}{ll}
		v-\min\{\hat\Gamma^{II}_{\phi_{x_0}}(t)|t\in[v, \phi_{x_0}(u-x_0)]\},& \ v<\phi_{x_0}\big(u-x_0\big),\\
		u-x_0-\min\{\hat\Gamma^{I}_{\phi_{x_0}}(t)|t\in[\phi_{x_0}(u-x_0),v]\},&  \ v\ge\phi_{x_0}\big(u-x_0\big), 
	\end{array}\right.
\end{eqnarray*}
where the second equality follows from the expressions of  $\Gamma_{\phi}$ and $\phi$ given in 	\eqref{eq:G}. 
Therefore, the function $B_{\Gamma_{\phi_{x_0}}}$ defined by \eqref{eq:Bx} is equal to $W$-ordinal sum $\big(\langle 0,x_0, W \rangle, \langle x_0, 1, B_{\Gamma_{\phi}} \rangle\big)^W$. Moreover, since both $B_{\Gamma_{\phi}}$
and $W$ are copulas, then  $B_{\Gamma_{\phi_{x_0}}}=\big(\langle 0,x_0, W \rangle, \langle x_0, 1, B_{\Gamma_{\phi}} \rangle\big)^W$ 
is also a copula. By Theorem \ref{prop:G}, its sub-track section coincides with $\Gamma_{\phi_{x_0}}$. Therefore, $B_{\Gamma_{\phi_{x_0}}}\in \mathbb{C}_{\Gamma_{\phi_{x_0}}}$

The symmetrization of $B_{\Gamma_{\phi}}$ 
is defined, for every $(u,v)\in[0,1]^2$, as  $B^*_{\Gamma_{\phi_{x_0}}}(u,v)=B_{\Gamma_{\phi_{x_0}}}(\max\{u,v\},\min\{u,v\})$. It can be verified easily  that  $(\max\{u,v\},\min\{u,v\})\in D_{L,x_0}$ holds if and only if $(u,v)\in [0,1]^2 \backslash S_{x_0}$. Applying Lemma \ref{lem:sym}, we know that $B^*_{\Gamma_{\phi_{x_0}}}$ is a symmetric quasi-copula. Additionally, 
$$B^*_{\Gamma_{\phi_{x_0}}}\big(x_0+\phi^{-1}_{x_0}(t),t\big)=B_{\Gamma_{\phi_{x_0}}}\big(\max\{x_0+\phi^{-1}_{x_0}(t),t\},\min\{x_0+\phi^{-1}_{x_0}(t),t\}\big)$$
$$\ \ \ \ \ \ \ \ \ \ \ \ \ \ \ \ \ \ \ \ \ \ \ \ \ \ \ \ \ \ \  =B_{\Gamma_{\phi_{x_0}}}\big(x_0+\phi^{-1}_{x_0}(t),t\big)=\Gamma_{\phi_{x_0}}(t),\ \ t\in[0,1-x_0],$$
where the second equality follows from the assumption that $\phi_{x_0}(t)\le t$. Consequently,  $B^*_{\Gamma_{\phi_{x_0}}}$
is a symmetric quasi-copula and its sub-track  section along the $x_0$-curve $\phi_{x_0}$ is exactly  $\Gamma_{\phi_{x_0}}$, i.e., $B^*_{\Gamma_{\phi_{x_0}}}\in \mathbb{Q}_{\Gamma_{\phi_{x_0}}}$. 
\end{proof}

We now show with an example that  $B^*_{\Gamma_{\phi_{x_0}}}$ is generally a proper quasi-copula. Specifically, we show that it is possible to obtain  $V_{B^*_{\Gamma_{\phi_{x_0}}}}([u,v]\times[u,v])<0$, for some $(u,v)\in[0,1]^2$ with $u<v$. 

\begin{example} 	For a fixed  $0<x_0\le \frac{1-\beta^{1/(1-\beta)}}{3-\beta^{1/(1-\beta)}}$, $\beta > 1$, let $\Gamma_{\phi_{x_0}}(t)=t$ be the sub-track section of the Fr\'{e}chet-Hoeffding upper bound $M$  along the $x_0$-curve $\phi_{x_0}(t)=(1-x_0)\big(\frac{t}{1-x_0}\big)^{\beta}$, $t\in[0, 1-x_0]$. For the value of $x_0$ under consideration, choose $u$ and $v$ such that  $x_0+(1-x_0)\beta^{1/(1-\beta)}\le  u<v\le u+x_0$. 
 	In this case, for all $t\in[ \phi_{x_0}(u-x_0), v]$, 
	$$\big(\phi^{-1}_{x_0}(t)-\Gamma_{ \phi_{x_0}}(t)\big)'=\frac{1}{\beta}t^{1/\beta-1}(1-x_0)^{1-1/\beta}-1\le 0.$$
Then, $\phi^{-1}_{x_0}(t)-\Gamma_{ \phi_{x_0}}(t)$ is nonincreasing on $[ \phi_{x_0}(u-x_0), v]$. Thus, 
	\begin{eqnarray*}
	&& V_{B^*_{\Gamma_{\phi_{x_0}}}}([u,v]\times[u,v])\nonumber\\
	&=& B^*_{\Gamma_{\phi_{x_0}}}(u,u)+ B^*_{\Gamma_{\phi_{x_0}}}(v,v)-B^*_{\Gamma_{\phi_{x_0}}}(u,v)-B^*_{\Gamma_{\phi_{x_0}}}(v,u)\nonumber \\
	&=&u-x_0-\min\big\{\phi^{-1}_{x_0}(t)-\Gamma_{\phi_{x_0}}(t)|t\in[ \phi_{x_0}(u-x_0), u] \big\}+v-x_0-\min\big\{\phi^{-1}_{x_0}(t)-\Gamma_{\phi_{x_0}}(t)|t\in[ \phi_{x_0}(v-x_0), v] \big\} \nonumber \\
	&&-2v+2x_0+2\min\big\{\phi^{-1}_{x_0}(t)-\Gamma_{\phi_{x_0}}(t)|t\in[ \phi_{x_0}(v-x_0), u] \big\}\nonumber\\
	&=&u-v+\big(u^{1/\beta}(1-x_0)^{1-1/\beta}-u\big)-\big(v^{1/\beta}(1-x_0)^{1-1/\beta}-v\big)\nonumber\\
	&=&(u^{1/\beta}-v^{1/\beta})(1-x_0)^{1-1/\beta}<0,
\end{eqnarray*}	
where the third equality follows from the nonincreasing property of $\phi^{-1}_{x_0}(t)-\Gamma_{ \phi_{x_0}}(t)$  on $[ \phi_{x_0}(u-x_0), v]$, and the inequality holds due to $u<v$.  
\end{example}

Using the (quasi-)copulas  $A_{\Gamma_{\phi}}$ and $K_{\Gamma_{\phi}}$ presented in \eqref{eq:A copula} and \eqref{eq:K copula} respectively, 
an argument analogous to that used in Proposition \ref{prop:Bertino copula*} yields the following two families of 
(quasi-)copulas whose sub-track section along the $x_0$-curve $\phi_{x_0}$ coincides with  $\Gamma_{\phi_{x_0}}$. 

\begin{prop}\label{prop:Kcopula*}
For a fixed $x_0\in(0,1)$, let  $\Gamma_{\phi_{x_0}}$ be a $\phi_{x_0}$-function, where   $\phi_{x_0}\in\Phi_{x_0}$ satisfies  $\phi_{x_0}(t)\le t$ for all $t\in[0,1-x_0]$.
	
	\noindent ($i$) If\ \   $2t-\Gamma_{\phi_{x_0}}\Big(\phi_{x_0}\big(t(1-x_0)\big)\Big)/(1-x_0)$ and $2t-\Gamma_{\phi_{x_0}}\big(t(1-x_0)\big)/(1-x_0)$ are nondecreasing
	on $[0, 1-x_0]$, then the function  $K_{\Gamma_{\phi_{x_0}}}:$ $[0, 1]^2\rightarrow$ $[0, 1]$ defined by 
	\begin{equation}\label{eq:K copula0}K_{\Gamma_{\phi_{x_0}}}(u,v)
		=\left\{\begin{array}{ll}
			\min\big\{u-x_0,v,\frac{\Gamma_{\phi_{x_0}}\big(\phi_{x_0}(u-x_0)\big)+\Gamma_{\phi_{x_0}}(v)}{2}\big\}, & \ {\rm if}\  (u,v)\in D_{L, x_0}, \\ 
			W(u,v),&  \ {\rm otherwise} 
		\end{array}\right.
	\end{equation}
is a copula belonging to $\mathbb{C}_{\Gamma_{\phi_{x_0}}}$, and its symmetrization  given by 
	\begin{equation}\label{eq:K copula*}K^*_{\Gamma_{\phi_{x_0}}}(u,v)
		=\left\{\begin{array}{ll}
			\min\big\{\max\{u,v\}-x_0,\min\{u,v\},\frac{\Gamma_{\phi_{x_0}}\big(\phi_{x_0}(\max\{u,v\}-x_0)\big)+\Gamma_{\phi_{x_0}}(\min\{u,v\})}{2}\big\}, & \ {\rm if}\  (u,v)\in [0,1]^2 \backslash S_{x_0} \\ 
			W(u,v),&  \ {\rm otherwise} 
		\end{array}\right.
	\end{equation}
is a symmetric quasi-copula belonging to $\mathbb{Q}_{\Gamma_{\phi_{x_0}}}$.
	
\noindent ($ii$) The function  $G_{\Gamma_{\phi_{x_0}}}:$ $[0, 1]^2\rightarrow$ $[0, 1]$ defined by \begin{equation}\label{eq:A copula0}G_{\Gamma_{\phi_{x_0}}}(u,v)
	=\left\{\begin{array}{ll}
		\min\Big\{v,u-x_0-\max\{\hat\Gamma^{I}_{\phi_{x_0}}(t)|t\in[v,\phi_{x_0}(u-x_0)]\}\Big\}, & {\rm if}\  (u,v)\in D_{L,x_0}\cap T_{L, \phi_{x_0}},\\ 
		\min\Big\{u-x_0, 	v-\max\{\hat\Gamma^{II}_{\phi_{x_0}}(t)|t\in[ \phi_{x_0}(u-x_0),v]\}\Big\},&  {\rm if}\  (u,v)\in D_{L,x_0}\cap T_{U, \phi_{x_0}}, \\
		W(u,v),&  \ {\rm otherwise} 
	\end{array}\right.
\end{equation}
	is a quasi-copula belonging to $\mathbb{Q}_{\Gamma_{\phi_{x_0}}}$, and its symmetrization  given by 
	\begin{equation}\label{eq:A copula*}G^*_{\Gamma_{\phi_{x_0}}}(u,v)
		=\left\{\begin{array}{ll}
			\min\Big\{\min\{u,v\},\max\{u,v\}-x_0- & {\rm if}\  (u,v)\in [0,1]^2 \backslash S_{x_0}\ \ {\rm and}\\ 
		\ \ \ \ \ \ \ \ \ \ 	\max\{\hat\Gamma^{I}_{\phi_{x_0}}(t)|t\in[\min\{u,v\},\phi_{x_0}(\max\{u,v\}-x_0)]\}\Big\}, &  \min\{u,v\}<\phi_{x_0}\big(\max\{u,v\}-x_0\big),\\
			\min\Big\{\max\{u,v\}-x_0, 	\min\{u,v\}-&  {\rm if}\  (u,v)\in [0,1]^2 \backslash S_{x_0}\ \ {\rm and}\\
			\ \ \ \ \ \ \ \ \ \ 	\max\{\hat\Gamma^{II}_{\phi_{x_0}}(t)|t\in[ \phi_{x_0}(\max\{u,v\}-x_0),\min\{u,v\}]\}\Big\},& \min\{u,v\}\ge \phi_{x_0}\big(\max\{u,v\}-x_0\big), \\
			W(u,v),&  \ {\rm otherwise} 
		\end{array}\right.
	\end{equation}
	is a symmetric quasi-copula belonging to  $\mathbb{Q}_{\Gamma_{\phi_{x_0}}}$.
\end{prop}

\begin{proof}
By arguments analogous to those in  the proof of Proposition \ref{prop:Bertino copula*}, we obtain that $K_{\Gamma_{\phi_{x_0}}}$ admits the $W$-ordinal sum representation $(\langle 0,x_0, W \rangle, \langle x_0, 1, K_{\Gamma_{\phi}} \rangle)^W$, where  $\Gamma_{\phi}$ is defined in \eqref{eq:G}, and that  $K^*_{\Gamma_{\phi_{x_0}}}$  
	is its symmetrization. Similarly,   $G_{\Gamma_{\phi_{x_0}}}$ is given by  
	$G_{\Gamma_{\phi_{x_0}}}=(\langle 0,x_0, W \rangle, \langle x_0, 1, A_{\Gamma_{\phi}} \rangle)^W$ with $\Gamma_{\phi}$ as in \eqref{eq:G}, and   $G^*_{\Gamma_{\phi_{x_0}}}$
	denotes its symmetrization.  Theorem \ref{prop:G} then ensures that $K_{\Gamma_{\phi_{x_0}}}$ is a copula whenever the required nondecreasing condition holds,  while   $G_{\Gamma_{\phi_{x_0}}}$ is a quasi-copula,  whose sub-track  section along the  $x_0$-curve $\phi_{x_0}$ is exactly $\Gamma_{\phi_{x_0}}$.   Lemma \ref{lem:sym} implies that both the symmetrizations $K^*_{\Gamma_{\phi_{x_0}}}$ and $G^*_{\Gamma_{\phi_{x_0}}}$ are symmetric  quasi-copulas with the same sub-track section $\Gamma_{\phi_{x_0}}$  along $\phi_{x_0}$. 
\end{proof}

It is worth noting that, although the $W$-ordinal sum $K_{\Gamma_{\phi_{x_0}}}$ is a copula, its symmetrization  $K^*_{\Gamma_{\phi_{x_0}}}$ is generally a proper quasi-copula. Likewise,   $G^*_{\Gamma_{\phi_{x_0}}}$ is generally a proper quasi-copula.

\begin{example} For a fixed $x_0\in(0,1/3]$, let  $\Gamma_{\phi_{x_0}}(t)=t$ be the sub-track of the Fr\'{e}chet-Hoeffding upper bound $M$  along the $x_0$-curve $\phi_{x_0}(t)=(1-x_0)\big(\frac{t}{1-x_0}\big)^{\beta}$, $t\in[0, 1-x_0]$, where $\beta>1$ is a constant chosen sufficiently close to 1. Consider a  square $[u,v]^2\subseteq [0,1]^2$ satisfying that $x_0\le u<v\le u+x_0$, $u-x_0\le  \frac{\frac{(u-x_0)^{\beta}}{(1-x_0)^{\beta-1}}+u}{2}$ and  $v-x_0\le  \frac{\frac{(v-x_0)^{\beta}}{(1-x_0)^{\beta-1}}+u}{2}$. A straightforward calculation yields that 
	\begin{eqnarray*}
		&& V_{K^*_{\Gamma_{\phi_{x_0}}}}([u,v]\times[u,v])\nonumber\\
	&=&K^*_{\Gamma_{\phi_{x_0}}}(u,u) + K^*_{\Gamma_{\phi_{x_0}}}(v,v)- K^*_{\Gamma_{\phi_{x_0}}}(u,v)- K^*_{\Gamma_{\phi_{x_0}}}(v,u) \nonumber \\
		&=&\min\big\{u-x_0,\frac{\Gamma_{\phi_{x_0}}\big(\phi_{x_0}(u-x_0)\big)+\Gamma_{\phi_{x_0}}(u)}{2}\big\}+ \min\big\{v-x_0,\frac{\Gamma_{\phi_{x_0}}\big(\phi_{x_0}(v-x_0)\big)+\Gamma_{\phi_{x_0}}(v)}{2}\big\} \nonumber \\ 	
		&&-2\min\big\{v-x_0,\frac{\Gamma_{\phi_{x_0}}\big(\phi_{x_0}(v-x_0)\big)+\Gamma_{\phi_{x_0}}(u)}{2}\big\}\nonumber\\ 
		&=&\min\big\{u-x_0,\frac{\frac{(u-x_0)^{\beta}}{(1-x_0)^{\beta-1}}+u}{2}\big\}+ \min\big\{v-x_0,\frac{\frac{(v-x_0)^{\beta}}{(1-x_0)^{\beta-1}}+v}{2}\big\}-2\min\big\{v-x_0,\frac{\frac{(v-x_0)^{\beta}}{(1-x_0)^{\beta-1}}+u}{2}\big\}\nonumber\\
		&=&u-v<0,
	\end{eqnarray*}	
since  $u<v$.  
\end{example}


\begin{example} Let $x_0\in(0,1/3]$ and $\Gamma_{\phi_{x_0}}(t)=t$ be the sub-track of the Fr\'{e}chet-Hoeffding upper bound $M$  along the $x_0$-curve $\phi_{x_0}(t)=(1-x_0)\big(\frac{t}{1-x_0}\big)^{\beta}$, $t\in[0, 1-x_0]$, $\beta\ge 1$. 
	Consider an arbitrary square $[u,v]^2\subseteq [0,1]^2$ satisfying that $x_0\le u<v\le u+x_0$. Since $\hat\Gamma^{II}_{\phi_{x_0}}(t)=t-\Gamma_{\phi_{x_0}}(t)=0$ for every $t\in[0, 1-x_0]$, a straightforward calculation yields that, for all $x_0\le u<v\le u+x_0$, 
	\begin{eqnarray*}
	V_{G^*_{\Gamma_{\phi_{x_0}}}}([u,v]\times[u,v])	&=&G^*_{\Gamma_{\phi_{x_0}}}(u,u) + G^*_{\Gamma_{\phi_{x_0}}}(v,v)- G^*_{\Gamma_{\phi_{x_0}}}(u,v)- G^*_{\Gamma_{\phi_{x_0}}}(v,u) \nonumber \\
		&=&\min\big\{u-x_0,u\big\}+ \min\big\{v-x_0,v\}-2 \min\big\{v-x_0,u\} \nonumber \\
		&=&u-v<0.
	\end{eqnarray*}	
\end{example}

\section{A family of symmetric copulas with a prescribed sub-track section}

In general, quasi-copulas with a prescribed sub-track section constructed via the symmetrization of a $W$-ordinal sum are proper quasi-copulas. 
We now address the construction of a symmetric copula $C$, rather than a quasi-copula, whose sub-track section along the $x_0$-curve $\phi_{x_0}$ coincides with 
$\Gamma_{\phi_{x_0}}$, where $x_0\in(0, 1/2]$. Although the resulting copula  is symmetric, it is not obtained by symmetrizing a $W$-ordinal
sum. Therefore, the conditions of a copula must be established by a direct verification. 

\begin{thm}\label{thm1}
	For a fixed $x_0\in(0, 1/2]$, let  $\Gamma_{\phi_{x_0}}$  be a $\phi_{x_0}$-function with  $\phi_{x_0}\in\Phi_{x_0}$ satisfying  $\phi_{x_0}(t)\le t$ for all $t\in[0,1-x_0]$.  Suppose that the function  
	$\Gamma_{\phi_{x_0}}(t)-\Gamma_{\phi_{x_0}}\big(\phi_{x_0}(t-x_0)\big)$ is nondecreasing on  $[x_0, 1-x_0]$, and $t+\hat \Gamma^I_{\phi_{x_0}}\big(\phi_{x_0}(t)\big)$ as well as $t+\hat \Gamma^{II}_{\phi_{x_0}}(t)$ are nondecreasing on  $[0, 1-x_0]$. Then, the function $C_{\Gamma_{\phi_{x_0}}}:$ $[0, 1]^2\rightarrow$ $[0, 1]$ defined by 
	\begin{equation}\label{eq:C copula*}C_{\Gamma_{\phi_{x_0}}}(u,v)
		=\left\{\begin{array}{ll}
			\min\Big\{u,v,	\frac{\Gamma_{\phi_{x_0}}\big(\phi_{x_0}( \max\{u,v\}-x_0)\big)+\Gamma_{\phi_{x_0}}(\min\{u,v\})}{2}\Big\}, & {\rm if}\  (u,v)\in [0,1]^2 \backslash S_{x_0},\\ 
	\frac{\Gamma_{\phi_{x_0}}(\min\{u,v\})}{2},&  {\rm if}\  (u,v)\in S_{L, x_0},\\
			\min\{H_{\phi_{x_0}}(u,v), H_{\phi_{x_0}}(v,u)\},& {\rm if}\  (u,v)\in S_{U, x_0}, 
		\end{array}\right.
	\end{equation}	
	where 
	\begin{equation}\label{eq:C copula*1}
		H_{\phi_{x_0}}(u,v):=u+\frac{\Gamma_{\phi_{x_0}}\big(\phi_{x_0}(v-x_0)\big)-(1-x_0)}{2},  
	\end{equation}
	is a symmetric copula whose sub-track section along the  $x_0$-curve $\phi_{x_0}$ coincides with $\Gamma_{\phi_{x_0}}$, that is,  $C_{\Gamma_{\phi_{x_0}}}\in\mathbb{C}_{\Gamma_{\phi_{x_0}}}$.	
\end{thm}

\begin{proof}
It is obvious that  the $x_0$-curve $\phi_{x_0}$, which connects the points $(x_0, 0)$ and $(1, 1-x_0)$, is entirely contained in the region $[0,1]^2 \backslash S_{x_0}$. It follows, by the property (S1) and the assumption $\phi_{x_0}(t)\le t$ for all $t\in[0,1-x_0]$, that
	$$C_{\Gamma_{\phi_{x_0}}}\big(x_0+\phi^{-1}_{x_0}(t),t\big)=\min\Big\{x_0+\phi^{-1}_{x_0}(t),t,	\frac{\Gamma_{\phi_{x_0}}\Big(\phi_{x_0}\big(x_0+\phi^{-1}_{x_0}(t)-x_0\big)\Big)+\Gamma_{\phi_{x_0}}(t)}{2}\Big\}=\min\big\{t,	\Gamma_{\phi_{x_0}}(t)\big\}=\Gamma_{\phi_{x_0}}(t).$$
 Thus, the last claim of the theorem  is established.

From the expression of $C_{\Gamma_{\phi_{x_0}}}$, it is immediate that   $C_{\Gamma_{\phi_{x_0}}}(u,v)=C_{\Gamma_{\phi_{x_0}}}(v,u)$ for all $(u,v)\in[0,1]^2$, and thus $C_{\Gamma_{\phi_{x_0}}}$ is symmetric on $[0,1]^2$.
	
We now prove that $C_{\Gamma_{\phi_{x_0}}}$ 
	is indeed a copula. The first boundary condition  in Definition  \ref{df:copula} is
	easily verified. The second condition, however, is less straightforward and  requires a separate  argument. 
	\begin{itemize}
		\item To this end, assume first that $v < 1-x_0$. Applying the property (S2), it 
		yields that
		$$\Gamma_{\phi_{x_0}}(1-x_0)-\Gamma_{\phi_{x_0}}(v)\le 1-x_0-v+\phi^{-1}_{x_0}(1-x_0)-\phi^{-1}_{x_0}(v).$$
		Using the equality $\Gamma_{\phi_{x_0}}(1-x_0)=1-x_0$ and the assumption $\phi_{x_0}(t)\le t$, this inequality is rewritten as 
		$$2v\le \phi^{-1}_{x_0}(v)+v\le \Gamma_{\phi_{x_0}}(v) +1-x_0. $$
		Then, 
		$$C_{\Gamma_{\phi_{x_0}}}(1,v)=\min\Big\{1,v, 	\frac{1-x_0+\Gamma_{\phi_{x_0}}(v)}{2}\Big\}=v.$$
		\item Next, we turn to the case $v \ge 1-x_0$. Under this assumption,
		$$C_{\Gamma_{\phi_{x_0}}}(1,v)=	\min\{H_{\phi_{x_0}}(1,v), H_{\phi_{x_0}}(v,1)\}.$$
		
		Notice that for all $(u,v)\in S_{U,x_0}$,
		\begin{equation}\label{ieq:1} H_{\phi_{x_0}}(u,v)\le  H_{\phi_{x_0}}(v,u)\end{equation} if and only if 
		$$u-\frac{\Gamma_{\phi_{x_0}}\big(\phi_{x_0}(u-x_0)\big)}{2}\le v-\frac{\Gamma_{\phi_{x_0}}\big(\phi_{x_0}(v-x_0)\big)}{2}. $$
		As a consequence of the assumption that $t+\hat \Gamma^{I}_{\phi_{x_0}}\big(\phi_{x_0}(t)\big)=2t-\Gamma_{\phi_{x_0}}\big(\phi_{x_0}(t)\big) $ is nondecreasing on  $[0, 1-x_0]$, it follows that,  for all $1-x_0\le u\le v\le 1$, the inequality \eqref{ieq:1} holds. 
		Hence, 	$$C_{\Gamma_{\phi_{x_0}}}(1,v)=H_{\phi_{x_0}}(v,1)= v+\frac{\Gamma_{\phi_{x_0}}\big(\phi_{x_0}(1-x_0)\big)-(1-x_0)}{2}=v.$$
	\end{itemize}

Since the boundary conditions have been established for  $C_{\Gamma_{\phi_{x_0}}}$, the only remaining requirement is to verify the 2-increasing property, i.e., the condition
(C2) of Definition \ref{df:copula}. To this end,  several cases need to be examined.	
	
\noindent \textit{Case 1}: We now consider a rectangle $[u_1,  u_2]\times [v_1, v_2]$ that is completely contained either in  $T_{L,\phi_{x_0}}$ or in 
 the region bounded by the lines $u=x_0$, $v=1-x_0$,  $v=u$ and the $x_0$-curve $v=\phi_{x_0}(u-x_0)$. Then, the rectangle is entirely contained in  $[0,1]^2 \backslash S_{x_0}$, and thus 
	\begin{eqnarray*}
	&& V_{C_{\Gamma_{\phi_{x_0}}}}([u_1,u_2]\times[v_1,v_2])\nonumber \\
	&=&	\min\Big\{v_1,	\frac{\Gamma_{\phi_{x_0}}\big(\phi_{x_0}(u_1-x_0)\big)+\Gamma_{\phi_{x_0}}(v_1)}{2}\Big\}+\min\Big\{v_2,	\frac{\Gamma_{\phi_{x_0}}\big(\phi_{x_0}(u_2-x_0)\big)+\Gamma_{\phi_{x_0}}(v_2)}{2}\Big\} \nonumber \\
	&&-\min\Big\{v_1,	\frac{\Gamma_{\phi_{x_0}}\big(\phi_{x_0}(u_2-x_0)\big)+\Gamma_{\phi_{x_0}}(v_1)}{2}\Big\}-\min\Big\{v_2,	\frac{\Gamma_{\phi_{x_0}}\big(\phi_{x_0}(u_1-x_0)\big)+\Gamma_{\phi_{x_0}}(v_2)}{2}\Big\},\nonumber
\end{eqnarray*}
since $v_i\le u_j$ for all $i,j\in\{1,2\}$. 

The proof proceeds by considering the following sub-cases:

\noindent (a) $v_1\le \frac{\Gamma_{\phi_{x_0}}\big(\phi_{x_0}(u_1-x_0)\big)+\Gamma_{\phi_{x_0}}(v_1)}{2}\le \frac{\Gamma_{\phi_{x_0}}\big(\phi_{x_0}(u_2-x_0)\big)+\Gamma_{\phi_{x_0}}(v_1)}{2}$: 
$$V_{C_{\Gamma_{\phi_{x_0}}}}([u_1,u_2]\times[v_1,v_2])=\min\Big\{v_2,	\frac{\Gamma_{\phi_{x_0}}\big(\phi_{x_0}(u_2-x_0)\big)+\Gamma_{\phi_{x_0}}(v_2)}{2}\Big\}-\min\Big\{v_2,	\frac{\Gamma_{\phi_{x_0}}\big(\phi_{x_0}(u_1-x_0)\big)+\Gamma_{\phi_{x_0}}(v_2)}{2}\Big\}\ge 0$$
	since $\Gamma_{\phi_{x_0}}(\phi_{x_0}(t))$ is nondecreasing with respect to $t$ and $u_2\ge u_1$. 
	
\noindent (b) $ \frac{\Gamma_{\phi_{x_0}}\big(\phi_{x_0}(u_1-x_0)\big)+\Gamma_{\phi_{x_0}}(v_1)}{2}< v_1\le \frac{\Gamma_{\phi_{x_0}}\big(\phi_{x_0}(u_2-x_0)\big)+\Gamma_{\phi_{x_0}}(v_1)}{2}$, then the following cases can be distinguished:
\begin{itemize}
	\item [(b1)] It is worth noting that the two inequalities  $\frac{\Gamma_{\phi_{x_0}}\big(\phi_{x_0}(u_1-x_0)\big)+\Gamma_{\phi_{x_0}}(v_1)}{2}< v_1$ and $v_2<\frac{\Gamma_{\phi_{x_0}}\big(\phi_{x_0}(u_1-x_0)\big)+\Gamma_{\phi_{x_0}}(v_2)}{2}$ cannot hold simultaneously. Indeed, these
	two inequalities imply 
	$$v_2-v_1<\frac{\Gamma_{\phi_{x_0}}\big(\phi_{x_0}(u_1-x_0)\big)+\Gamma_{\phi_{x_0}}(v_2)}{2}-\frac{\Gamma_{\phi_{x_0}}\big(\phi_{x_0}(u_1-x_0)\big)+\Gamma_{\phi_{x_0}}(v_1)}{2}=\frac{\Gamma_{\phi_{x_0}}(v_2)-\Gamma_{\phi_{x_0}}(v_1)}{2},$$
	which contradicts the assumption that $t+\hat \Gamma^{II}_{\phi_{x_0}}(t)=2t-\Gamma_{\phi_{x_0}}(t)$ is nondecreasing on  $[0, 1-x_0]$. Hence, the following cases remain:
	\item [(b2)] If $\frac{\Gamma_{\phi_{x_0}}\big(\phi_{x_0}(u_1-x_0)\big)+\Gamma_{\phi_{x_0}}(v_2)}{2}\le v_2<\frac{\Gamma_{\phi_{x_0}}\big(\phi_{x_0}(u_2-x_0)\big)+\Gamma_{\phi_{x_0}}(v_2)}{2}$, then, by the nondecreasing property of   $t+\hat \Gamma^{II}_{\phi_{x_0}}(t)$   on  $[0, 1-x_0]$, it follows that 
	$$V_{C_{\Gamma_{\phi_{x_0}}}}([u_1,u_2]\times[v_1,v_2])=	\frac{\Gamma_{\phi_{x_0}}\big(\phi_{x_0}(u_1-x_0)\big)+\Gamma_{\phi_{x_0}}(v_1)}{2}-	\frac{\Gamma_{\phi_{x_0}}\big(\phi_{x_0}(u_1-x_0)\big)+\Gamma_{\phi_{x_0}}(v_2)}{2}+v_2-v_1\ge 0.$$
		\item [(b3)] If $\frac{\Gamma_{\phi_{x_0}}\big(\phi_{x_0}(u_1-x_0)\big)+\Gamma_{\phi_{x_0}}(v_2)}{2}\le \frac{\Gamma_{\phi_{x_0}}\big(\phi_{x_0}(u_2-x_0)\big)+\Gamma_{\phi_{x_0}}(v_2)}{2}\le v_2$, then we have
	$$V_{C_{\Gamma_{\phi_{x_0}}}}([u_1,u_2]\times[v_1,v_2])=	\frac{\Gamma_{\phi_{x_0}}\big(\phi_{x_0}(u_2-x_0)\big)+\Gamma_{\phi_{x_0}}(v_1)}{2}-v_1\ge 0.$$
\end{itemize}

\noindent (c) $\frac{\Gamma_{\phi_{x_0}}\big(\phi_{x_0}(u_2-x_0)\big)+\Gamma_{\phi_{x_0}}(v_1)}{2}< v_1$. Then, it must hold that  $\frac{\Gamma_{\phi_{x_0}}\big(\phi_{x_0}(u_2-x_0)\big)+\Gamma_{\phi_{x_0}}(v_2)}{2}\le v_2$. Indeed, otherwise the inequalities $\frac{\Gamma_{\phi_{x_0}}\big(\phi_{x_0}(u_2-x_0)\big)+\Gamma_{\phi_{x_0}}(v_1)}{2}< v_1$ 
and $\frac{\Gamma_{\phi_{x_0}}\big(\phi_{x_0}(u_2-x_0)\big)+\Gamma_{\phi_{x_0}}(v_2)}{2}> v_2$ imply
$$v_2-v_1<\frac{\Gamma_{\phi_{x_0}}(v_2)-\Gamma_{\phi_{x_0}}(v_1)}{2},$$
which contradicts the assumption that $t+\hat \Gamma^{II}_{\phi_{x_0}}(t)$ is nondecreasing on  $[0, 1-x_0]$. Hence, 
	\begin{eqnarray*}
V_{C_{\Gamma_{\phi_{x_0}}}}([u_1,u_2]\times[v_1,v_2])\nonumber	&=&		\frac{\Gamma_{\phi_{x_0}}\big(\phi_{x_0}(u_1-x_0)\big)+\Gamma_{\phi_{x_0}}(v_1)}{2}+	\frac{\Gamma_{\phi_{x_0}}\big(\phi_{x_0}(u_2-x_0)\big)+\Gamma_{\phi_{x_0}}(v_2)}{2} \nonumber \\
	&&-	\frac{\Gamma_{\phi_{x_0}}\big(\phi_{x_0}(u_2-x_0)\big)+\Gamma_{\phi_{x_0}}(v_1)}{2}-	\frac{\Gamma_{\phi_{x_0}}\big(\phi_{x_0}(u_1-x_0)\big)+\Gamma_{\phi_{x_0}}(v_2)}{2}=0.\nonumber
\end{eqnarray*}

Consequently, in this case we always have   $V_{C_{\Gamma_{\phi_{x_0}}}}([u_1,u_2]\times[v_1,v_2])\ge 0$.	
	
\noindent \textit{Case 2}:  Consider the rectangle $[u_1, u_2] \times[\phi_{x_0}(u_1-x_0), \phi_{x_0}(u_2-x_0)]$, where $u_1$ and $u_2$ satisfy $1\ge u_1\ge x_0\ge u_2-u_1>0$. 
Then, under the assumption that $\phi_{x_0}(t)\le t$, we have that for $i=1,2$, 
$$C_{\Gamma_{\phi_{x_0}}}\big(u_i, \phi_{x_0}(u_i-x_0)\big)=\min\Big\{\phi_{x_0}(u_i-x_0),\frac{\Gamma_{\phi_{x_0}}\big(\phi_{x_0}(u_i-x_0)\big)+\Gamma_{\phi_{x_0}}\big(\phi_{x_0}(u_i-x_0)\big)}{2}\Big\}=\Gamma_{\phi_{x_0}}\big(\phi_{x_0}(u_i-x_0)\big),$$
where the second equality follows from the property (S1)
of the $\phi_{x_0}$-function. Hence,  
	\begin{eqnarray*}
		&& V_{C_{\Gamma_{\phi_{x_0}}}}\big([u_1,u_2]\times[\phi_{x_0}(u_1-x_0), \phi_{x_0}(u_2-x_0)]\big)\nonumber \\
		&=&	\Gamma_{\phi_{x_0}}\big(\phi_{x_0}(u_1-x_0)\big)+\Gamma_{\phi_{x_0}}\big(\phi_{x_0}(u_2-x_0)\big)-\min\Big\{u_1, \phi_{x_0}(u_2-x_0),  	\frac{\Gamma_{\phi_{x_0}}\big(\phi_{x_0}(u_1-x_0)\big)+\Gamma_{\phi_{x_0}}\big(\phi_{x_0}(u_2-x_0)\big)}{2}\Big\} \nonumber \\
		&&-\min\Big\{u_2, \phi_{x_0}(u_1-x_0),	\frac{\Gamma_{\phi_{x_0}}\big(\phi_{x_0}(u_2-x_0)\big)+\Gamma_{\phi_{x_0}}\big(\phi_{x_0}(u_1-x_0)\big)}{2}\Big\},\nonumber\\
		&\ge & \frac{\Gamma_{\phi_{x_0}}\big(\phi_{x_0}(u_1-x_0)\big)+\Gamma_{\phi_{x_0}}\big(\phi_{x_0}(u_2-x_0)\big)}{2}-	\frac{\Gamma_{\phi_{x_0}}(\phi_{x_0}(u_2-x_0))+\Gamma_{\phi_{x_0}}(\phi_{x_0}(u_1-x_0))}{2}= 0.
	\end{eqnarray*}	
\noindent \textit{Case 3}: Consider a square $[u_1,u_2]\times[u_1,u_2]\subseteq[0,1]^2
\textbackslash S_{x_0}$, where $x_0 \le u_1 < u_2 \le 1-x_0$. Then 
	\begin{eqnarray*}
		&& V_{C_{\Gamma_{\phi_{x_0}}}}([u_1,u_2]\times[u_1,u_2])\nonumber \\
		&=& \min\Big\{u_1,  	\frac{\Gamma_{\phi_{x_0}}\big(\phi_{x_0}(u_1-x_0)\big)+\Gamma_{\phi_{x_0}}(u_1)}{2}\Big\}+ \min\Big\{u_2,  	\frac{\Gamma_{\phi_{x_0}}\big(\phi_{x_0}(u_2-x_0)\big)+\Gamma_{\phi_{x_0}}(u_2)}{2}\Big\} \nonumber \\
		&&-2\min\Big\{u_1,	\frac{\Gamma_{\phi_{x_0}}\big(\phi_{x_0}(u_2-x_0)\big)+\Gamma_{\phi_{x_0}}(u_1)}{2}\Big\}\nonumber\\
		&= & \frac{\Gamma_{\phi_{x_0}}\big(\phi_{x_0}(u_1-x_0)\big)+\Gamma_{\phi_{x_0}}(u_1)}{2}+\frac{\Gamma_{\phi_{x_0}}\big(\phi_{x_0}(u_2-x_0)\big)+\Gamma_{\phi_{x_0}}(u_2)}{2}\nonumber\\
			&&-2\min\Big\{u_1,	\frac{\Gamma_{\phi_{x_0}}\big(\phi_{x_0}(u_2-x_0)\big)+\Gamma_{\phi_{x_0}}(u_1)}{2}\Big\}\nonumber\\
		&\ge & \frac{\Gamma_{\phi_{x_0}}\big(\phi_{x_0}(u_1-x_0)\big)+\Gamma_{\phi_{x_0}}(u_1)}{2}+\frac{\Gamma_{\phi_{x_0}}\big(\phi_{x_0}(u_2-x_0)\big)+\Gamma_{\phi_{x_0}}(u_2)}{2}-\Big(\Gamma_{\phi_{x_0}}\big(\phi_{x_0}(u_2-x_0)\big)+\Gamma_{\phi_{x_0}}(u_1)\Big)\nonumber\\
		&=& \frac{\Gamma_{\phi_{x_0}}(u_2)-\Gamma_{\phi_{x_0}}\big(\phi_{x_0}(u_2-x_0)\big)}{2}-\frac{\Gamma_{\phi_{x_0}}(u_1)-\Gamma_{\phi_{x_0}}(\phi_{x_0}\big(u_1-x_0)\big)}{2}\ge 0, 
	\end{eqnarray*}	
where the second equality holds due to that   $\Gamma_{\phi_{x_0}}(t)\le t$ and $\phi_{x_0}(t)\le t$ for every $t\in[0,1-x_0]$, and the second inequality follows from the assumption that $\Gamma_{\phi_{x_0}}(t)-\Gamma_{\phi_{x_0}}\big(\phi_{x_0}(t-x_0)\big)$ is nondecreasing on $[x_0,1-x_0]$.
	
\noindent \textit{Case 4}: Consider a rectangle $[u_1,  u_2]\times [v_1, v_2]\subseteq S_{L, x_0}$ with $u_1\ge v_2$, that is, $0 \le v_1 < v_2 \le u_1 <  u_2 \le x_0$.
	Then 
		$$ V_{C_{\Gamma_{\phi_{x_0}}}}([u_1,  u_2]\times [v_1, v_2])=\frac{\Gamma_{\phi_{x_0}}(\min\{u_1,v_1\})}{2}+\frac{\Gamma_{\phi_{x_0}}(\min\{u_2,v_2\})}{2}-\frac{\Gamma_{\phi_{x_0}}(\min\{u_2,v_1\})}{2}-\frac{\Gamma_{\phi_{x_0}}(\min\{u_1,v_2\})}{2}=0.$$

\noindent \textit{Case 5}: Consider a square $[u_1,  u_2]\times [u_1, u_2]\subseteq S_{L, x_0}$. Then
	\begin{eqnarray*}
	V_{C_{\Gamma_{\phi_{x_0}}}}([u_1,  u_2]\times [u_1, u_2])	&=& \frac{\Gamma_{\phi_{x_0}}(\min\{u_1,u_1\})}{2}+\frac{\Gamma_{\phi_{x_0}}(\min\{u_2,u_2\})}{2}-\frac{\Gamma_{\phi_{x_0}}(\min\{u_2,u_1\})}{2}-\frac{\Gamma_{\phi_{x_0}}(\min\{u_1,u_2\})}{2}\nonumber\\
		&=& \frac{\Gamma_{\phi_{x_0}}(u_2)}{2}-\frac{\Gamma_{\phi_{x_0}}(u_1)}{2}\ge 0. 
	\end{eqnarray*}	

\noindent \textit{Case 6}: Consider a rectangle $[u_1,  u_2]\times [v_1, v_2]\subseteq S_{U, x_0}$ with $u_1\ge v_2$, that is, $1-x_0\le v_1 < v_2 \le u_1 < u_2 \le 1$. Then, from the relation \eqref{ieq:1} between $H_{\phi_{x_0}}(u,v)$ and $H_{\phi_{x_0}}(v,u)$ for $(u,v)\in S_{U, x_0}$, it yields that 
	\begin{eqnarray*}
	V_{C_{\Gamma_{\phi_{x_0}}}}([u_1,  u_2]\times [v_1, v_2])	
	&=& H_{\phi_{x_0}}(v_2,u_2)+H_{\phi_{x_0}}(v_1,u_1)-H_{\phi_{x_0}}(v_2,u_1)-H_{\phi_{x_0}}(v_1,u_2)\nonumber \\
&=&	v_2+\frac{\Gamma_{\phi_{x_0}}\big(\phi_{x_0}(u_2-x_0)\big)-(1-x_0)}{2}+v_1+\frac{\Gamma_{\phi_{x_0}}\big(\phi_{x_0}(u_1-x_0)\big)-(1-x_0)}{2}
		\nonumber \\
		&&-v_2-\frac{\Gamma_{\phi_{x_0}}\big(\phi_{x_0}(u_1-x_0)\big)-(1-x_0)}{2}-v_1-\frac{\Gamma_{\phi_{x_0}}\big(\phi_{x_0}(u_2-x_0)\big)-(1-x_0)}{2}=0.\nonumber 
	\end{eqnarray*}	

\noindent \textit{Case 7}: Consider a square $[u_1,  u_2]\times [u_1, u_2]\subseteq  S_{U, x_0}$.
Then, 
	\begin{eqnarray*}
	V_{C_{\Gamma_{\phi_{x_0}}}}([u_1,  u_2]\times [u_1, u_2])	&=& H_{\phi_{x_0}}(u_2,  u_2)+H_{\phi_{x_0}}(u_1,  u_1)-2H_{\phi_{x_0}}(u_1,  u_2)
		\nonumber \\
		&=& u_2-\frac{\Gamma_{\phi_{x_0}}\big(\phi_{x_0}(u_2-x_0)\big)}{2}-u_1+\frac{\Gamma_{\phi_{x_0}}\big(\phi_{x_0}(u_1-x_0)\big)}{2}\ge 0, \nonumber 
	\end{eqnarray*}	
where the  inequality follows from the assumption that
$t+\hat \Gamma^{I}_{\phi_{x_0}}\big(\phi_{x_0}(t)\big)=2t-\Gamma_{\phi_{x_0}}\big(\phi_{x_0}(t)\big)$ is  nondecreasing on  $[0, 1-x_0]$.

Taking into account the symmetry of  $C_{\Gamma_{\phi_{x_0}}}$, any rectangle $R\subset [0,1]^ 
2$ whose vertices lie in more than one of the  
considered regions can be divided into a finite number  of rectangles $R_i$, each belonging to one of the seven cases above. Then,  $V_{C_{\Gamma_{\phi_{x_0}}}}(R)$ equals the sum of  $V_{C_{\Gamma_{\phi_{x_0}}}}(R_i)$. Therefore, it holds that 
 $V_{C_{\Gamma_{\phi_{x_0}}}}(R)\ge 0$. 

This completes the proof.	\end{proof}

\begin{example}
	\label{EX: sc}  	Let $x_0=0.45$ and $\Gamma_{\phi_{x_0}}(t)=t$ be the sub-track section of the Fr\'{e}chet-Hoeffding upper bound $M$  along the $x_0$-curve $\phi_{x_0}(t)=\frac{t^2}{0.55}$, $t\in[0, 0.55]$. It can be verify that the function  
	$\Gamma_{\phi_{x_0}}(t)-\Gamma_{\phi_{x_0}}\big(\phi_{x_0}(t-x_0)\big)=t-\frac{(t-0.45)^2}{0.55}$ is nondecreasing on  $[0.45, 0.55]$, and $t+\hat \Gamma^I_{\phi_{x_0}}\big(\phi_{x_0}(t)\big)=2t-\frac{t^2}{0.55}$ as well as $t+\hat \Gamma^{II}_{\phi_{x_0}}(t)=t$ are nondecreasing on  $[0, 0.55]$. Therefore, the assumptions in Theorem \ref{thm1} are satisfied, and we obtain a symmetric copula $C_{\Gamma_{\phi_{x_0}}}$ with the prescribed  sub-track section   $\Gamma_{\phi_{x_0}}$,   expressed as 
		\begin{equation*}C_{\Gamma_{\phi_{x_0}}}(u,v)
		=\left\{\begin{array}{ll}
			\min\Big\{u,v,	\frac{\frac{(\max\{u,v\}-0.45)^2}{0.55}+\min\{u,v\}}{2}\Big\}, & {\rm if}\  (u,v)\in [0,1]^2 \backslash S_{x_0},\\ 
			\frac{\min\{u,v\}}{2},&  {\rm if}\  (u,v)\in S_{L, x_0},\\
			\min\{u+\frac{0.55}{2}\big((\frac{v-0.45}{0.55})^{2}-1\big), v+\frac{0.55}{2}\big((\frac{u-0.45}{0.55})^{2}-1\big)\},& {\rm if}\  (u,v)\in S_{U, x_0}. 
		\end{array}\right.
	\end{equation*}	
\end{example}

\section{The best-possible bounds for the class   $\mathbb{Q}_{\Gamma_{\phi_{x_0}}}$}

 For a prescribed sub-track section $\Gamma_{\phi_{x_0}}$, we have verified that the classes  $\mathbb{Q}_{\Gamma_{\phi_{x_0}}}$
and $\mathbb{C}_{\Gamma_{\phi_{x_0}}}$
are nonempty. In this section, we address the second problem stated in Introduction, that is, the determination of best-possible bounds for $\mathbb{Q}_{\Gamma_{\phi_{x_0}}}$. 


We begin by establishing the best-possible lower bound for the class  $\mathbb{Q}_{\Gamma_{\phi_{x_0}}}$.

\begin{thm}\label{B+}
For a fixed $x_0\in (0,1)$ and a $\phi_{x_0}$-function $\Gamma_{\phi_{x_0}}$ with $\phi_{x_0}\in\Phi_{x_0}$ satisfying   $\phi_{x_0}(t)\le t$ for all $t\in[0,1-x_0]$, the copula $B_{\Gamma_{\phi_{x_0}}}$
given by \eqref{eq:Bx} is the smallest (quasi-)copula
whose sub-track section along the $x_0$-curve $\phi_{x_0}$ is exactly $\Gamma_{\phi_{x_0}}$, i.e.,  $B_{\Gamma_{\phi_{x_0}}}\le Q$ for all $Q\in \mathbb{Q}_{\Gamma_{\phi_{x_0}}}$.
\end{thm}

\begin{proof}
As established in Proposition \ref{prop:Bertino copula*}, the function $B_{\Gamma_{\phi_{x_0}}}$ 
	is a copula and therefore a quasi-copula; moreover, $B_{\Gamma_{\phi_{x_0}}}\in \mathbb{C}_{\Gamma_{\phi_{x_0}}}\subset \mathbb{Q}_{\Gamma_{\phi_{x_0}}}$. Now 
	let $Q$ be any (quasi-)copula with the prescribed sub-track section $\Gamma_{\phi_{x_0}}$ along the $x_0$-curve $\phi_{x_0}$. Clearly,  $B_{\Gamma_{\phi_{x_0}}}(u, v)=W(u,v)\le  Q(u, v)$ for every $(u, v)\notin D_{L, x_0}$.  Now consider 
	 $(u,v)\in D_{L, x_0}$ and assume first that $v \ge \phi_{x_0}(u-x_0)$,  
	thus $B_{\Gamma_{\phi_{x_0}}}(u,v)=u-x_0-\hat \Gamma^{I}_{\phi_{x_0}}(t^*)$, where  $$\hat \Gamma^{I}_{\phi_{x_0}}(t^*):=\min\{\phi^{-1}_{x_0}(t)-\Gamma_{\phi_{x_0}}(t)| t\in[\phi_{x_0}(u-x_0),v] \}.$$
	
We recall  \citep[see][Proposition 3]{Genest1999A} that a quasi-copula satisfies the inequality \eqref{qe:V} whenever at least one of the arguments $u_1$, $u_2$, $v_1$ or $v_2$  
equals  either 0 or 1. Applying this inequality for $Q$ to the rectangle $[u, \phi^{-1}_{x_0}(t^*)+x_0]\times[t^*,1]$, it  
 yields that
$$0\le \phi^{-1}_{x_0}(t^*)+x_0-u+Q(u,t^*)-Q\big(\phi^{-1}_{x_0}(t^*)+x_0,t^*\big)=\phi^{-1}_{x_0}(t^*)-(u-x_0)+Q(u,t^*)-\Gamma_{\phi_{x_0}}(t^*),$$
and thus 
$$B_{\Gamma_{\phi_{x_0}}}(u,v)=u-x_0-\hat \Gamma^{I}_{\phi_{x_0}}(t^*)=\Gamma_{\phi_{x_0}}(t^*)-\big(\phi^{-1}_{x_0}(t^*)-(u-x_0)\big)\le Q(u,t^*)\le Q(u,v).$$
	
An analogous argument holds when $v < \phi_{x_0}(u-x_0)$. In this case, $B_{\Gamma_{\phi_{x_0}}}(u,v)=v-\hat \Gamma^{II}_{\phi_{x_0}}(t^\star)$, where  $$\hat \Gamma^{II}_{\phi_{x_0}}(t^\star):=\min\{t-\Gamma_{\phi_{x_0}}(t)|t\in[v, \phi_{x_0}(u-x_0)] \}.$$  
Applying the inequality  \eqref{qe:V} for $Q$ to the rectangle  $[\phi^{-1}_{x_0}(t^\star)+x_0, 1]\times[v,t^\star]$, it  
 yields 
$$0\le t^\star -v+Q\big(\phi^{-1}_{x_0}(t^\star)+x_0,v\big)-Q\big(\phi^{-1}_{x_0}(t^\star)+x_0, t^\star\big)=t^\star -v+Q\big(\phi^{-1}_{x_0}(t^\star)+x_0,v\big)-\Gamma_{\phi_{x_0}}(t^\star),$$
and then 
$$B_{\Gamma_{\phi_{x_0}}}(u,v)=v-\hat \Gamma^{II}_{\phi_{x_0}}(t^\star)=\Gamma_{\phi_{x_0}}(t^\star)-\big(t^\star-v\big)\le Q\big(\phi^{-1}_{x_0}(t^\star)+x_0,v\big)\le Q(u,v).$$
This completes the proof.
\end{proof}

In the following theorem, we establish  the best-possible upper bound for the class  $\mathbb{Q}_{\Gamma_{\phi_{x_0}}}$.

\begin{thm}\label{A}
	For a fixed $x_0\in (0,1)$ and a $\phi_{x_0}$-function $\Gamma_{\phi_{x_0}}$ with $\phi_{x_0}\in\Phi_{x_0}$ satisfying  $\phi_{x_0}(t)\le t$ for all $t\in[0,1-x_0]$, 	the function  $A_{\Gamma_{\phi_{x_0}}}:$ $[0,1]^2\rightarrow$ $[0,1]$ 
given by 
	\begin{equation}\label{eq:G copula1*}A_{\Gamma_{\phi_{x_0}}}(u,v)
	=\left\{\begin{array}{ll}
	\min\{u, v-q^{II}_{\phi_{x_0}}(u,v)\}, & \ {\rm if}\  (u,v)\in T_{U,\phi_{x_0}}, \\ 
	\min\{v, u-x_0-q^{I}_{\phi_{x_0}}(u,v)\},&  \ {\rm if}\  (u,v)\in T_{L,\phi_{x_0}}  
	\end{array}\right.
	\end{equation}	
	is a quasi-copula whose sub-track section along the $x_0$-curve $\phi_{x_0}$ equals   $\Gamma_{\phi_{x_0}}$. Moreover, $Q\le A_{\Gamma_{\phi_{x_0}}}$ for all $Q\in \mathbb{Q}_{\Gamma_{\phi_{x_0}}}$.
\end{thm}

The proof of this theorem will be carried out through several auxiliary propositions and lemmas. 
To begin with, we prove that    $A_{\Gamma_{\phi_{x_0}}}$ is bounded within the unit interval.

\begin{lem}\label{lem24}
	 For a fixed $x_0\in(0,1)$ and a $\phi_{x_0}$-function  $\Gamma_{\phi_{x_0}}$, it holds that   $v-q^{II}_{\phi_{x_0}}(u,v)\ge 0$ for all $(u,v)\in T_{U,\phi_{x_0}}$, and  $u-x_0-q^{I}_{\phi_{x_0}}(u,v)\ge 0$ for all $(u,v)\in T_{L,\phi_{x_0}}$. 
	 \end{lem}
	 
\begin{proof}
For any $(u,v)\in T_{L,\phi_{x_0}}$, define $\hat \Gamma^{I}_{\phi_{x_0}}(t_1^*)=q^I_{\phi_{x_0}}(u,v)$, where $v\le t_1^* \le \phi_{x_0}(u-x_0)$, and for any $(u,v)\in T_{U,\phi_{x_0}}$, define $\hat \Gamma^{II}_{\phi_{x_0}}(t_2^*)=q^{II}_{\phi_{x_0}}(u,v)$, where $\phi_{x_0}(\max\{u-x_0,0\})\le t_2^* \le v$. From the property (D3), we know that for any $(u,v)\in T_{L,\phi_{x_0}}$,  
	$$q^I_{\phi_{x_0}}(u,v)= \hat \Gamma^{I}_{\phi_{x_0}}(t_1^*)\le \phi^{-1}_{x_0}(t_1^*)\le u-x_0,  $$
	and for any $(u,v)\in T_{U,\phi_{x_0}}$,  	$$q^{II}_{\phi_{x_0}}(u,v)= \hat \Gamma^{II}_{\phi_{x_0}}(t_2^*)\le t_2^* \le v.  $$
Using the above two inequalities, we obtain the desired result.
\end{proof}

\begin{lem}\label{lem:1}
For the function  $A_{\Gamma_{\phi_{x_0}}}$  defined by \eqref{eq:G copula1*}, it holds that   for every  $u\in[0,1]$, 
	$$A_{\Gamma_{\phi_{x_0}}}(u,0)=A_{\Gamma_{\phi_{x_0}}}(0,u)=0\ \ {\rm and}\ \ A_{\Gamma_{\phi_{x_0}}}(u,1)=A_{\Gamma_{\phi_{x_0}}}(1,u)=u.$$
\end{lem}

\begin{proof}
	If $(u,v)\in T_{U,\phi_{x_0}}$ and either $u=0$ or $v=0$, Lemma \ref{lem24} implies  that $A_{\Gamma_{\phi_{x_0}}}(u,0)=A_{\Gamma_{\phi_{x_0}}}(0,v)=0$, since $v-q^{II}_{\phi_{x_0}}(u,v)\ge 0$  for all $(u,v)\in T_{U,\phi_{x_0}}$. We now turn to the case $(u,v)\in T_{L,\phi_{x_0}}$ with $v=0$, which means $u\ge x_0$ and $v=0$. Then, the inequality  $u-x_0-q^{I}_{\phi_{x_0}}(u,v)\ge 0$ implies $A_{\Gamma_{\phi_{x_0}}}(u,0)=0$.
	
	We now consider the boundary cases in which $u=1$ or $v=1$. 
	\begin{itemize}
	\item First, consider the case $u=1$. 
	\begin{itemize}
		\item In the case where $u=1$ and $v\ge 1-x_0$, from the expression \eqref{eq:G copula1*} of $A_{\Gamma_{\phi_{x_0}}}$, we have $A_{\Gamma_{\phi_{x_0}}}(1,v)=\min\{1, v-q^{II}_{\phi_{x_0}}(1,v)\}=v$, since $q^{II}_{\phi_{x_0}}(1,v)=\max\{t-\Gamma_{\phi_{x_0}}(t)|t\in[1-x_0,1-x_0]\}=1-x_0-\Gamma_{\phi_{x_0}}(1-x_0)=0$.
		\item  In the case where $u=1$ and $v < 1-x_0$, from \eqref{eq:G copula1*}, we know that $A_{\Gamma_{\phi_{x_0}}}(1,v)=\min\{v,  1-x_0-q^{I}_{\phi_{x_0}}(1,v)\}$. It follows from the property (D3) that   
		$$ \hat\Gamma^I_{\phi_{x_0}}(t)\le \min\{\phi^{-1}_{x_0}(t),1-x_0-t\}\le 1-x_0-t$$
		for all $t\in[v,1-x_0]$. Therefore,  $ \hat\Gamma^I_{\phi_{x_0}}(t)\le 1-x_0-v$ or, equivalently, $v\le 1-x_0-\hat\Gamma^I_{\phi_{x_0}}(t)$ for all $t\in [v, \phi_{x_0}(1-x_0)]$. 
		Consequently, 
		$$v\le 1-x_0-\max\{\hat\Gamma^I_{\phi_{x_0}}(t)|t\in[v,\phi_{x_0}(1-x_0)]\}=1-x_0-q^{I}_{\phi_{x_0}}(1,v).$$
		Then, for all $v<1-x_0$,  $A_{\Gamma_{\phi_{x_0}}}(1,v)=\min\{v,  1-x_0-q^{I}_{\phi_{x_0}}(1,v)\}=v$.
	\end{itemize}

\item Next consider the case $v=1$. For any $u\in[0,1]$, we have $(u,1)\in T_{U,\phi_{x_0}}$, and then $$A_{\Gamma_{\phi_{x_0}}}(u,1)=\min\{u, 1-q^{II}_{\phi_{x_0}}(u,1)\}.$$
\begin{itemize}
	\item Assume first that $u\ge x_0$. From (D3), it follows  that $ \hat\Gamma^{II}_{\phi_{x_0}}(t)\le \min\{t,1-x_0-\phi^{-1}_{x_0}(t)\}\le 1-x_0-\phi^{-1}_{x_0}(t)$
	for each $t\in[0, 1-x_0]$. Therefore, for all $\phi_{x_0}(u-x_0)\le t\le 1-x_0 $, 
	$$\hat\Gamma^{II}_{\phi_{x_0}}(t)\le 1-x_0-\phi^{-1}_{x_0}(t)\le 1-x_0-(u-x_0)=1-u, $$
	which is equivalent to $u\le 1- \hat\Gamma^{II}_{\phi_{x_0}}(t)$ for all $t \in[\phi_{x_0}(u-x_0), 1-x_0]$. Consequently, we have
	$$u\le 1-\max\{\hat\Gamma^{II}_{\phi_{x_0}}(t)|t\in[\phi_{x_0}(u-x_0),1-x_0]\}=1-q^{II}_{\phi_{x_0}}(u,1),$$
	which implies $A_{\Gamma_{\phi_{x_0}}}(u,1)=u$.
	\item For $u< x_0$, we have $m_{\phi_{x_0}}(u,1)=0$ and
	$$1-q^{II}_{\phi_{x_0}}(u,1)=1-\max\{\hat\Gamma^{II}_{\phi_{x_0}}(t)|t\in[0,1-x_0]\}\ \ \ \ \ \ \ \ \ \ \ \ \ \ \ \ \ \ \ \ \ \ \ \ \ \ \ \ \ \ \ \ \ \ \ \ \ \ \ \ \ \ \ \ \ \ \ \ \ \ \ \ \ \ \ \ \ \ \ \    $$
	$$\ge 1-\max\{t|t\in[0,1-x_0]\}+\min\{\Gamma_{\phi_{x_0}}(t)|t\in[0,1-x_0]\}$$
	$$=1-(1-x_0)+\Gamma_{\phi_{x_0}}(0)=x_0> u.    \ \ \ \ \ \ \ \ \ \ \ \ \ \ \ \ \ \ \ \ \ \ \ \ \ \ \ \ \ \ \ \ \ \ \ \ \    $$
	Therefore,   $A_{\Gamma_{\phi_{x_0}}}(u,1)=u$ also holds in this case.
\end{itemize}
\end{itemize}
\end{proof}

\begin{prop}
	The function  $A_{\Gamma_{\phi_{x_0}}}$ defined by \eqref{eq:G copula1*} is nondecreasing in each argument
	and satisfies the Lipschitz condition \eqref{Lp}.
\end{prop}

\begin{proof}
	In order to show the nondecreasing property and the Lipschitz condition for  $A_{\Gamma_{\phi_{x_0}}}$, 
	it suffices to verify that  $A_{\Gamma_{\phi_{x_0}}}$ 
	is nondecreasing and satisfies 
	Lipschitz condition separately on $T_{U,\phi_{x_0}}$ and $T_{L,\phi_{x_0}}$. The proofs for other cases can be reduced to the two  situations considered above, using either the triangle inequality or basic summation. Accordingly, we distinguish four cases and, in each, prove the required 
	properties. 
	
	\noindent \textit{Case} 1: Consider the case that both $(u_1, v)$ and $(u_2, v)$ are in $T_{U,\phi_{x_0}}$ with $u_1\le  u_2$. Then,  $\phi_{x_0}(\max\{u_1-x_0,0\})\le \phi_{x_0}(\max\{u_2-x_0,0\})\le v$, and consequently $m_{\phi_{x_0}}(u_1,v)\le m_{\phi_{x_0}}(u_2,v)$ while  $M_{\phi_{x_0}}(u_1,v)= M_{\phi_{x_0}}(u_2,v)$. It implies 
$q^{II}_{\phi_{x_0}}(u_1,v)\ge q^{II}_{\phi_{x_0}}(u_2,v)$. Recall that for every   $(u, v)\in T_{U,\phi_{x_0}}$,
$A_{\Gamma_{\phi_{x_0}}}(u,v)=\min\{u, v-q^{II}_{\phi_{x_0}}(u,v)\}$. 

\begin{itemize}
	\item We first show   that, for each fixed $v$, the mapping $u \mapsto A_{\Gamma_{\phi_{x_0}}}(u,v)$ is nondecreasing. If $A_{\Gamma_{\phi_{x_0}}}(u_2,v)=u_2$, then
	$$A_{\Gamma_{\phi_{x_0}}}(u_2,v)-A_{\Gamma_{\phi_{x_0}}}(u_1,v)=u_2-\min\{u_1, v-q^{II}_{\phi_{x_0}}(u_1,v)\}\ge u_2-u_1\ge 0.$$
	If $A_{\Gamma_{\phi_{x_0}}}(u_2,v)=v-q^{II}_{\phi_{x_0}}(u_2,v)$, then
	$$A_{\Gamma_{\phi_{x_0}}}(u_2,v)-A_{\Gamma_{\phi_{x_0}}}(u_1,v)=v-q^{II}_{\phi_{x_0}}(u_2,v)-\min\{u_1, v-q^{II}_{\phi_{x_0}}(u_1,v)\}\ge v-q^{II}_{\phi_{x_0}}(u_2,v)-\big( v-q^{II}_{\phi_{x_0}}(u_1,v)\big)\ge 0.$$
	\item We now turn to the Lipschitz condition. If $A_{\Gamma_{\phi_{x_0}}}(u_1,v)=u_1$, then
	$$A_{\Gamma_{\phi_{x_0}}}(u_2,v)-A_{\Gamma_{\phi_{x_0}}}(u_1,v)=\min\{u_2, v-q^{II}_{\phi_{x_0}}(u_2,v)\}-u_1\le u_2-u_1.$$
	If $A_{\Gamma_{\phi_{x_0}}}(u_1,v)=v-q^{II}_{\phi_{x_0}}(u_1,v)$, then
	$$A_{\Gamma_{\phi_{x_0}}}(u_2,v)-A_{\Gamma_{\phi_{x_0}}}(u_1,v)\le  v-q^{II}_{\phi_{x_0}}(u_2,v)-\big(v-q^{II}_{\phi_{x_0}}(u_1,v)\big)\le q^{II}_{\phi_{x_0}}(u_1,v)-q^{II}_{\phi_{x_0}}(u_2,v).$$
	It therefore remains to verify the inequality $q^{II}_{\phi_{x_0}}(u_1,v)-q^{II}_{\phi_{x_0}}(u_2,v)\le u_2-u_1$.
	
	If $q^{II}_{\phi_{x_0}}(u_1,v)=q^{II}_{\phi_{x_0}}(u_2,v)$, the claim is trivial. Then, suppose that $q^{II}_{\phi_{x_0}}(u_1,v)>q^{II}_{\phi_{x_0}}(u_2,v)$.  It implies that $m_{\phi_{x_0}}(u_2,v)>m_{\phi_{x_0}}(u_1,v)\ge 0$, $m_{\phi_{x_0}}(u_1,v)=\phi_{x_0}(\max\{u_1-x_0,0\})$ and 
	$m_{\phi_{x_0}}(u_2,v)=\phi_{x_0}(u_2-x_0)$. Consequently, there exist $t'$ and $t''$ satisfying  
	$$m_{\phi_{x_0}}(u_1,v)\le t'<m_{\phi_{x_0}}(u_2,v)\le t''\le M_{\phi_{x_0}}(u_2,v)  $$
	such that $q^{II}_{\phi_{x_0}}(u_1,v)=\hat\Gamma^{II}_{\phi_{x_0}}(t')$ and $q^{II}_{\phi_{x_0}}(u_2,v)=\hat\Gamma^{II}_{\phi_{x_0}}(t'')$. Since $\hat\Gamma^{II}_{\phi_{x_0}}(t'')$ is the maximum of $\hat\Gamma^{II}_{\phi_{x_0}}$ 
	on the interval $[m_{\phi_{x_0}}(u_2,v),$
	$ M_{\phi_{x_0}}(u_2,v)]$, we have  $\hat\Gamma^{II}_{\phi_{x_0}}(t'')\ge \hat\Gamma^{II}_{\phi_{x_0}}(t)$ for every $t$ in that interval, and 
	in particular for $t =m_{\phi_{x_0}}(u_2,v)= \phi_{x_0}(u_2-x_0)$. Applying the property (D2) of $\hat\Gamma^{II}_{\phi_{x_0}}$, we obtain 
	$$q^{II}_{\phi_{x_0}}(u_1,v)-q^{II}_{\phi_{x_0}}(u_2,v)=\hat\Gamma^{II}_{\phi_{x_0}}(t')-\hat\Gamma^{II}_{\phi_{x_0}}(t'')\le \hat\Gamma^{II}_{\phi_{x_0}}(t')-\hat\Gamma^{II}_{\phi_{x_0}}\big(\phi_{x_0}(u_2-x_0)\big)\le u_2-x_0-\phi^{-1}_{x_0}(t') $$
	$$\le u_2-x_0-\max\{u_1-x_0,0\}\le u_2-u_1,\ \ \ \ \ \ \ \ \ \ \ \ \ \ \ \   $$
	which establishes the claim in this case. 
	
\end{itemize}

 \noindent \textit{Case} 2:   Let both $(u, v_1)$ and $(u, v_2)$ be in $T_{U,\phi_{x_0}}$ with $v_1\le  v_2$, meaning that $\phi_{x_0}(\max\{u-x_0,0\})\le v_1\le v_2$. In this case, we know that 
 $$m_{\phi_{x_0}}(u,v_1)=m_{\phi_{x_0}}(u,v_2)=\phi_{x_0}(\max\{u-x_0,0\})$$
 and 
 $$M_{\phi_{x_0}}(u,v_1)=v_1\le v_2=M_{\phi_{x_0}}(u,v_2).$$
Thus, 
 $q^{II}_{\phi_{x_0}}(u,v_1)\le q^{II}_{\phi_{x_0}}(u,v_2)$. Recall that for any $(u, v)\in T_{U,\phi_{x_0}}$,  $A_{\Gamma_{\phi_{x_0}}}(u,v)=\min\{u, v-q^{II}_{\phi_{x_0}}(u,v)\}$. 
 
 \begin{itemize}
 	\item To verify the nondecreasing property of $A_{\Gamma_{\phi_{x_0}}}$,  assume first that  $A_{\Gamma_{\phi_{x_0}}}(u,v_2)=u$.  Then
 	$$A_{\Gamma_{\phi_{x_0}}}(u,v_2)-A_{\Gamma_{\phi_{x_0}}}(u,v_1)=u-\min\{u, v_1-q^{II}_{\phi_{x_0}}(u,v_1)\}\ge u-u= 0.$$
 	If $A_{\Gamma_{\phi_{x_0}}}(u,v_2)=v_2-q^{II}_{\phi_{x_0}}(u,v_2)$, then
 	$$A_{\Gamma_{\phi_{x_0}}}(u,v_2)-A_{\Gamma_{\phi_{x_0}}}(u,v_1)\ge v_2-q^{II}_{\phi_{x_0}}(u,v_2)-\big(v_1-q^{II}_{\phi_{x_0}}(u,v_1)\big)=v_2-v_1-\big(q^{II}_{\phi_{x_0}}(u,v_2)-q^{II}_{\phi_{x_0}}(u,v_1)\big) .$$ 
 	It therefore remains to verify the inequality  $v_2-v_1\ge q^{II}_{\phi_{x_0}}(u,v_2)-q^{II}_{\phi_{x_0}}(u,v_1)$. This inequality is obviously satisfied when
 	$q^{II}_{\phi_{x_0}}(u,v_2)=q^{II}_{\phi_{x_0}}(u,v_1)$.  In the case where  $q^{II}_{\phi_{x_0}}(u,v_2)>q^{II}_{\phi_{x_0}}(u,v_1)$,  we have $M_{\phi_{x_0}}(u,v_1) =v_1 < 1-x_0$, and there
 	exist $t'$ and $t''$ satisfying  $m_{\phi_{x_0}}(u,v_1)\le t'\le v_1<t''\le M_{\phi_{x_0}}(u,v_2)$ such that 
 	$$q^{II}_{\phi_{x_0}}(u,v_2)=\hat\Gamma^{II}_{\phi_{x_0}}(t'')>\hat\Gamma^{II}_{\phi_{x_0}}(t')=q^{II}_{\phi_{x_0}}(u,v_1).$$
 	Then, from the property (D2) of  $\hat\Gamma^{II}_{\phi_{x_0}}$
 	and the fact that $\hat\Gamma^{II}_{\phi_{x_0}}(t')$ is the maximum of $\hat\Gamma^{II}_{\phi_{x_0}}$ 
 	on the interval $[m_{\phi_{x_0}}(u,v_1), M_{\phi_{x_0}}(u,v_1)]$, we have
 	$$0\le q^{II}_{\phi_{x_0}}(u,v_2)-q^{II}_{\phi_{x_0}}(u,v_1)=\hat\Gamma^{II}_{\phi_{x_0}}(t'')-\hat\Gamma^{II}_{\phi_{x_0}}(t')\le \hat\Gamma^{II}_{\phi_{x_0}}(t'')-\hat\Gamma^{II}_{\phi_{x_0}}(v_1)$$
 	$$ \ \ \ \ \ \ \ \ \ \ \ \ \ \ \ \ \ \ \ \ \ \ \  \ \ \ \ \ \ \ \ \ \ \ \ \ \ \ \ \ \ \ \ \ \ \ \le t''-v_1\le\min\{v_2,1-x_0\}-v_1\le v_2-v_1.   $$	
 \item	We now turn to the Lipschitz condition of $A_{\Gamma_{\phi_{x_0}}}$. If $A_{\Gamma_{\phi_{x_0}}}(u,v_1)=u$, then  
 	$$A_{\Gamma_{\phi_{x_0}}}(u,v_2)-A_{\Gamma_{\phi_{x_0}}}(u,v_1)=\min\{u, v_2-q^{II}_{\phi_{x_0}}(u,v_2)\}-u\le u-u=0\le v_2-v_1.$$ 
 	In the case where  $A_{\Gamma_{\phi_{x_0}}}(u,v_1)=v_1-q^{II}_{\phi_{x_0}}(u,v_1)$, it follows that 
 	$$A_{\Gamma_{\phi_{x_0}}}(u,v_2)-A_{\Gamma_{\phi_{x_0}}}(u,v_1)\le v_2-q^{II}_{\phi_{x_0}}(u,v_2)-\big(v_1-q^{II}_{\phi_{x_0}}(u,v_1)\big)=v_2-v_1-\big(q^{II}_{\phi_{x_0}}(u,v_2)-q^{II}_{\phi_{x_0}}(u,v_1)\big)$$
 	$$\le v_2-v_1,\ \ \ \ \ \ \ \ \ \ \ \ \ \ \ \ \ \ \ \ \ \ \ \ \ \ \ \ \ \ \ \ \ \ \ \ \ \ \ \ \ \ \ \ \ \ \ \ \ \ \ \ \ \ \ \ \ \ \ \ \  \   $$ 
 	since $q^{II}_{\phi_{x_0}}(u,v_2)-q^{II}_{\phi_{x_0}}(u,v_1)\ge 0$. 
 \end{itemize}

\noindent \textit{Case} 3:  Consider the case that both $(u_1, v)$ and $(u_2, v)$ are in $T_{L,\phi_{x_0}}$ with $x_0\le u_1\le  u_2\le 1$, i.e., $v\le \phi_{x_0}(u_1-x_0) \le \phi_{x_0}(u_2-x_0)$. In this case, it holds that 
$$m_{\phi_{x_0}}(u_1,v)=m_{\phi_{x_0}}(u_2,v)=v$$
and 
$$M_{\phi_{x_0}}(u_1,v)=\phi_{x_0}(u_1-x_0) \le \phi_{x_0}(u_2-x_0)=M_{\phi_{x_0}}(u_2,v).$$
Thus,  
$q^{I}_{\phi_{x_0}}(u_1,v)\le q^{I}_{\phi_{x_0}}(u_2,v)$. Note that for any $(u, v)\in T_{L,\phi_{x_0}}$, $A_{\Gamma_{\phi_{x_0}}}(u,v)=\min\{v, u-x_0-q^{I}_{\phi_{x_0}}(u,v)\}$. 

\begin{itemize}
	\item For the nondecreasing property of $A_{\Gamma_{\phi_{x_0}}}$, if $A_{\Gamma_{\phi_{x_0}}}(u_2,v)=v$,  then
	$$A_{\Gamma_{\phi_{x_0}}}(u_2,v)-A_{\Gamma_{\phi_{x_0}}}(u_1,v)\ge v-v= 0.$$
	If $A_{\Gamma_{\phi_{x_0}}}(u_2,v)=u_2-x_0-q^{I}_{\phi_{x_0}}(u_2,v)$, then
	$$A_{\Gamma_{\phi_{x_0}}}(u_2,v)-A_{\Gamma_{\phi_{x_0}}}(u_1,v)\ge u_2-u_1-\big(q^{I}_{\phi_{x_0}}(u_2,v)-q^{I}_{\phi_{x_0}}(u_1,v)\big).$$ 
	If $q^{I}_{\phi_{x_0}}(u_2,v)=q^{I}_{\phi_{x_0}}(u_1,v)$, it follows immediately that $A_{\Gamma_{\phi_{x_0}}}(u_2,v)-A_{\Gamma_{\phi_{x_0}}}(u_1,v)\ge 0$. In the case where  $q^{I}_{\phi_{x_0}}(u_2,v)>q^{I}_{\phi_{x_0}}(u_1,v)$, 
	there necessarily exist $t'$ and $t''$ satisfying  
	$$v\le t'\le \phi_{x_0}(u_1-x_0)< t''\le \phi_{x_0}(u_2-x_0)  $$
	such that
	$q^{I}_{\phi_{x_0}}(u_2,v)=\hat\Gamma^{I}_{\phi_{x_0}}(t'')$ and $q^{I}_{\phi_{x_0}}(u_1,v)=\hat\Gamma^{I}_{\phi_{x_0}}(t')$. Applying an argument analogous to those used in the preceding cases, we have that
	$$q^{I}_{\phi_{x_0}}(u_2,v)-q^{I}_{\phi_{x_0}}(u_1,v)=\hat\Gamma^{I}_{\phi_{x_0}}(t'')-\hat\Gamma^{I}_{\phi_{x_0}}(t')\le \hat\Gamma^{I}_{\phi_{x_0}}(t'')-\hat\Gamma^{I}_{\phi_{x_0}}(\phi_{x_0}(u_1-x_0))$$
	$$\ \ \ \ \ \ \ \ \ \ \ \ \ \ \ \ \ \ \ \ \ \ \ \ \ \ \ \ \ \ \ \ \ \ \ \ \ \ \ \ \  \le \phi^{-1}_{x_0}(t'')- (u_1-x_0) \le u_2-x_0-(u_1-x_0)=u_2-u_1,$$
	where the second inequality follows from the property (D2) of $\hat\Gamma^{I}_{\phi_{x_0}}$. 
	Thus, $A_{\Gamma_{\phi_{x_0}}}$ 
	is nondecreasing in its first argument on $T_{L,\phi_{x_0}}$.
\item 	
	For the Lipschitz condition, if $A_{\Gamma_{\phi_{x_0}}}(u_1,v)=v$, then
	$$A_{\Gamma_{\phi_{x_0}}}(u_2,v)-A_{\Gamma_{\phi_{x_0}}}(u_1,v)\le v-v=0\le u_2-u_1.$$ 
	If $A_{\Gamma_{\phi_{x_0}}}(u_1,v)=u_1-x_0-q^{I}_{\phi_{x_0}}(u_1,v)$, then
	$$A_{\Gamma_{\phi_{x_0}}}(u_2,v)-A_{\Gamma_{\phi_{x_0}}}(u_1,v)\le u_2-u_1-\big(q^{I}_{\phi_{x_0}}(u_2,v)-q^{I}_{\phi_{x_0}}(u_1,v)\big)\le u_2-u_1,$$ 
	since $q^{I}_{\phi_{x_0}}(u_1,v)\le q^{I}_{\phi_{x_0}}(u_2,v)$. 
\end{itemize}

\noindent \textit{Case} 4: Finally, consider the case that  both $(u, v_1)$ and $(u, v_2)$ belong to  $T_{L,\phi_{x_0}}$ with $v_1\le  v_2$, i.e., $v_1\le v_2\le \phi_{x_0}(u-x_0)$. In this case, we know that 
$$m_{\phi_{x_0}}(u,v_1)= v_1\le  v_2=m_{\phi_{x_0}}(u,v_2)$$
and 
$$M_{\phi_{x_0}}(u,v_1)=\phi_{x_0}(u-x_0)=M_{\phi_{x_0}}(u,v_2).$$
Thus, 
$q^{I}_{\phi_{x_0}}(u,v_1)\ge q^{I}_{\phi_{x_0}}(u,v_2)$. If  $q^{I}_{\phi_{x_0}}(u,v_1)> q^{I}_{\phi_{x_0}}(u,v_2)$, then there
exist $t'$ and $t''$ satisfying $v_1\le t'< v_2\le t''\le \phi_{x_0}(u-x_0)$ such that 
$q^{I}_{\phi_{x_0}}(u,v_1)=\hat\Gamma^{I}_{\phi_{x_0}}(t')$ and $q^{I}_{\phi_{x_0}}(u,v_2)=\hat\Gamma^{I}_{\phi_{x_0}}(t'')$. Applying an argument similar to those used in the preceding cases yields 
$$0\le q^{I}_{\phi_{x_0}}(u,v_1)-q^{I}_{\phi_{x_0}}(u,v_2)=\hat\Gamma^{I}_{\phi_{x_0}}(t')-\hat\Gamma^{I}_{\phi_{x_0}}(t'')\le \hat\Gamma^{I}_{\phi_{x_0}}(t')-\hat\Gamma^{I}_{\phi_{x_0}}(v_2)\le v_2- t'\le v_2- v_1,$$
where the third inequality follows from the property (D2) of $\hat\Gamma^{I}_{\phi_{x_0}}$. 
Recall that  $A_{\Gamma_{\phi_{x_0}}}(u,v)=\min\{v, u-x_0-q^{I}_{\phi_{x_0}}(u,v)\}$ for all $(u,v)\in T_{L,\phi_{x_0}}$. 

\begin{itemize}
	\item For the nondecreasing property, if $A_{\Gamma_{\phi_{x_0}}}(u,v_2)=v_2$, then $A_{\Gamma_{\phi_{x_0}}}(u,v_2)-A_{\Gamma_{\phi_{x_0}}}(u,v_1)\ge v_2-v_1\ge 0$. Otherwise, $A_{\Gamma_{\phi_{x_0}}}(u,v_2)-A_{\Gamma_{\phi_{x_0}}}(u,v_1)\ge q^{I}_{\phi_{x_0}}(u,v_1)-q^{I}_{\phi_{x_0}}(u,v_2)\ge 0$. 
	\item For the Lipschitz condition, if $A_{\Gamma_{\phi_{x_0}}}(u,v_1)=u-x_0-q^{I}_{\phi_{x_0}}(u,v_1)$, then  
	$$A_{\Gamma_{\phi_{x_0}}}(u,v_2)-A_{\Gamma_{\phi_{x_0}}}(u,v_1)\le q^{I}_{\phi_{x_0}}(u,v_1)-q^{I}_{\phi_{x_0}}(u,v_2)\le v_2-v_1.$$ 
	Otherwise, the inequality  $A_{\Gamma_{\phi_{x_0}}}(u,v_2)-A_{\Gamma_{\phi_{x_0}}}(u,v_1)\le v_2-v_1$  follows immediately.
\end{itemize}

In summary, in every case considered, we have verified that $A_{\Gamma_{\phi_{x_0}}}$ 
is nondecreasing and satisfies the Lipschitz condition. \end{proof}

\begin{lem}
The sub-track section  of the quasi-copula  $A_{\Gamma_{\phi_{x_0}}}$ defined by \eqref{eq:G copula1*} along the $x_0$-curve $\phi_{x_0}$ coincides with $\Gamma_{\phi_{x_0}}$, that is, $A_{\Gamma_{\phi_{x_0}}}\in \mathbb{Q}_{\Gamma_{\phi_{x_0}}}$.
\end{lem}

\begin{proof} For any $t\in[0,1-x_0]$, we have
	$$m_{\phi_{x_0}}\big(x_0+\phi^{-1}_{x_0}(t),t\big)=t=M_{\phi_{x_0}}\big(x_0+\phi^{-1}_{x_0}(t),t\big).$$
Consequently, 	
	$$q^{II}_{\phi_{x_0}}\big(x_0+\phi^{-1}_{x_0}(t),t\big)=t-\Gamma_{\phi_{x_0}}(t),$$
and thus
	$$A_{\Gamma_{\phi_{x_0}}}\big(x_0+\phi^{-1}_{x_0}(t),t\big)=\min\{x_0+\phi^{-1}_{x_0}(t), t-q^{II}_{\phi_{x_0}}\big(x_0+\phi^{-1}_{x_0}(t),t\big)\}=\min\{x_0+\phi^{-1}_{x_0}(t),\Gamma_{\phi_{x_0}}(t)\}=\Gamma_{\phi_{x_0}}(t),$$
which completes the proof.
\end{proof}

\begin{prop}
	For a fixed $x_0\in (0,1)$ and a $\phi_{x_0}$-function $\Gamma_{\phi_{x_0}}$ with $\phi_{x_0}\in\Phi_{x_0}$ satisfying  $\phi_{x_0}(t)\le t$ for all $t\in[0,1-x_0]$, $A_{\Gamma_{\phi_{x_0}}}$
	is the largest quasi-copula whose sub-track section along the $x_0$-curve $\phi_{x_0}$ equals  $\Gamma_{\phi_{x_0}}$, i.e., $Q\le  A_{\Gamma_{\phi_{x_0}}}$
	for all $Q\in \mathbb{Q}_{\Gamma_{\phi_{x_0}}}$.
\end{prop}

\begin{proof}
 Consider a quasi-copula  $Q\in \mathbb{Q}_{\Gamma_{\phi_{x_0}}}$. For a given $(u, v) \in[0, 1]^2$, we distinguish two cases. 
 \begin{itemize}
 	\item  Assume first that
 	$(u, v) \in T_{U,\phi_{x_0}}$, that is, $v\ge \phi_{x_0} (\max\{u-x_0,0\})$. Then,
 	$A_{\Gamma_{\phi_{x_0}}}(u,v)=\min\{u, v-q^{II}_{\phi_{x_0}}(u,v)\}$ 
 	and 
 	$$q^{II}_{\phi_{x_0}}(u,v)=\max\{t-\Gamma_{\phi_{x_0}}(t)|t\in[m_{\phi_{x_0}}(u,v), M_{\phi_{x_0}}(u,v)]\}$$
 	$$\ \ \ \ \ \ \ \ \ \ \ \ \ \ \ \ \ \  \ \ \ \ \ \ \ \ \ \ \ \ \ \ \ \ \ \ \ \ \  \ \  =\max\big\{t-\Gamma_{\phi_{x_0}}(t)|t\in[ \phi_{x_0}(\max\{u-x_0,0\}),\min\{v, 1-x_0\}]\big\}.$$
 	Since $Q(u,v)\le \min\{u,v\}$, it suffices to verify the inequality $Q(u,v)\le v-q^{II}_{\phi_{x_0}}(u,v)$. For any $t\in[ \phi_{x_0}(\max\{u-x_0,0\}),\min\{v, 1-x_0\}]$, the Lipschitz condition yields
 	$$Q(u,v)-Q(u,t)\le v-t\ \ {\rm or,\  equivalently}\ \ Q(u,v)\le v-t+Q(u,t).$$
 	Since a quasi-copula is nondecreasing,	then	for all $t\in [ \phi_{x_0}(\max\{u-x_0,0\}),\min\{v,1-x_0\}]$, it holds that
 	$$Q(u,v)\le v-t+Q(u,t)\le v-t+Q\big(x_0+\phi^{-1}_{x_0}(t),t\big)=v-t+\Gamma_{\phi_{x_0}}(t)=v-\hat\Gamma^{II}_{\phi_{x_0}}(t),$$	
 	and thus,
 	$$Q(u,v)\le v-\max\big\{\hat \Gamma^{II}_{\phi_{x_0}}(t)|t\in[ \phi_{x_0}(\max\{u-x_0,0\}),\min\{v, 1-x_0\}]\big\}=v-q^{II}_{\phi_{x_0}}(u,v).$$	
 Consequently, $Q(u,v)\le  A_{\Gamma_{\phi_{x_0}}}(u,v)$ for all $(u, v)\in T_{U,\phi_{x_0}}$.  
 
 \item We now consider the case $(u, v)\in T_{L,\phi_{x_0}}$, which means $x_0\le u\le 1$ and $v\le \phi_{x_0} (u-x_0)$. Since $Q(u,v)\le \min\{u,v\}\le v$, it suffices to prove the inequality $Q(u, v)\le 
 u-x_0-q^{I}_{\phi_{x_0}}(u,v)$. 
 By an argument analogous to that used in the above case, we have that for every $t\in [m_{\phi_{x_0}}(u, v), M_{\phi_{x_0}}(u, v)]=[v, \phi_{x_0}(u-x_0)]$,
 $$Q(u,v)\le u-\big(\phi^{-1}_{x_0}(t)+x_0\big)+Q\big(\phi^{-1}_{x_0}(t)+x_0,v\big)\le u-x_0-\phi^{-1}_{x_0}(t)+Q\big(\phi^{-1}_{x_0}(t)+x_0,t\big)$$
 $$=u-x_0-\phi^{-1}_{x_0}(t)+\Gamma_{\phi_{x_0}}(t)=u-x_0-\hat\Gamma^{I}_{\phi_{x_0}}(t),\ \ \ \ \ \ \ \ \ \ \ \ \ \ \ \ \ \ \ \ \ \ \ \ \ \ \ \ \ \ \ \ \ \ \ \ \ \ \  $$
where the first inequality follows from the Lipschitz condition of $Q$, and the second inequality follows from the nondecreasing property of $Q$. 
 Hence, $Q(u,v)\le u-x_0-\max\{\hat\Gamma^{I}_{\phi_{x_0}}(t)|t\in[v, \phi_{x_0}(u-x_0)]\}=u-x_0-q^{I}_{\phi_{x_0}}(u,v)$, which implies $Q(u,v)\le  A_{\Gamma_{\phi_{x_0}}}(u,v)$ for all $(u,v)\in T_{L,\phi_{x_0}}$. 
 
 This completes the proof.\end{itemize}\end{proof}

The examples presented below show that $A_{\Gamma_{\phi_{x_0}}}$ is generally a proper quasi-copula and that it is the best-possible upper bound for quasi-copulas  with a prescribed  sub-track section $\Gamma_{\phi_{x_0}}$ along the $x_0$-curve $\phi_{x_0}$.

\begin{example} Let  $x_0=0.1$ and consider the sub-track section  $\Gamma_{\phi_{x_0}}(t)=t\big(x_0+t^{1/2}(1-x_0)^{1/2}\big)$ of the product copula $\Pi$  along the $x_0$-curve $\phi_{x_0}(t)=(1-x_0)\big(\frac{t}{1-x_0}\big)^{2}$, $t\in[0, 1-x_0]$.   Choose a  square $[x_0+0.01^{1/2}(1-x_0)^{1/2},\  x_0+0.02^{1/2}(1-x_0)^{1/2}]\times[0.01,0.02]=[0.19487,0.23416]\times[0.01,0.02]\subset [0,1]^2$. A straightforward calculation shows that  
	\begin{eqnarray*}&&	V_{A_{\Gamma_{\phi_{x_0}}}}([0.19487,0.23416]\times[0.01,0.02])	\nonumber \\	&=&A_{\Gamma_{\phi_{x_0}}}(0.19487,0.01) + A_{\Gamma_{\phi_{x_0}}}(0.23416,0.02)- A_{\Gamma_{\phi_{x_0}}}(0.19487,0.02)- A_{\Gamma_{\phi_{x_0}}}(0.23416,0.01)\nonumber \\
	&=&-0.00273<0.	\end{eqnarray*}
Hence, in this case, $A_{\Gamma_{\phi_{x_0}}}$ is not a copula but a proper quasi-copula. 	
\end{example}

\begin{example} 	Let $x_0\in(0,1)$ and $\Gamma_{\phi_{x_0}}(t)=t$ be the sub-track section of the Fr\'{e}chet-Hoeffding upper bound $M$  along the $x_0$-curve $\phi_{x_0}(t)=(1-x_0)\big(\frac{t}{1-x_0}\big)^{\beta}$, $t\in[0, 1-x_0]$, $\beta\ge 1$.
Notice easily that  $\hat\Gamma^{I}_{\phi_{x_0}}(t)=\phi^{-1}_{x_0}(t)-\Gamma_{\phi_{x_0}}(t)=t^{1/\beta}(1-x_0)^{1-1/\beta}-t$ and  $\hat\Gamma^{II}_{\phi_{x_0}}(t)=t-\Gamma_{\phi_{x_0}}(t)=0$ for every $t\in[0, 1-x_0]$. Hence, 
	\begin{equation*}A_{\Gamma_{\phi_{x_0}}}(u,v)
	=\left\{\begin{array}{ll}
		\min\{u, v\}, & \ {\rm if}\  (u,v)\in T_{U,\phi_{x_0}}, \\ 
		\min\big\{v, u-x_0-\max\{t^{1/\beta}(1-x_0)^{1-1/\beta}-t|t\in[v, \phi_{x_0}(u-x_0)]\}\big\},&  \ {\rm if}\  (u,v)\in T_{L,\phi_{x_0}}. 
	\end{array}\right.
\end{equation*}	
Notice that when $t=t_0:=(1-x_0)\beta^{\beta/(1-\beta)}$, the function $t^{1/\beta}(1-x_0)^{1-1/\beta}-t$ attains the maximum on $[0,1-x_0]$. Therefore, 
\begin{itemize}
	\item[($i$)] If $\phi_{x_0}(u-x_0)\le t_0$, then $\max\{t^{1/\beta}(1-x_0)^{1-1/\beta}-t|t\in[v, \phi_{x_0}(u-x_0)]\}=u-x_0-\phi_{x_0}(u-x_0)$ such that $ u-x_0-\max\{t^{1/\beta}(1-x_0)^{1-1/\beta}-t|t\in[v, \phi_{x_0}(u-x_0)]\}=u-x_0-\big(u-x_0-\phi_{x_0}(u-x_0)\big)=\phi_{x_0}(u-x_0)\ge v$.
	\item  [($ii$)] If $v< t_0< \phi_{x_0}(u-x_0)$, then $\max\{t^{1/\beta}(1-x_0)^{1-1/\beta}-t|t\in[v, \phi_{x_0}(u-x_0)]\}=\phi^{-1}_{x_0}(t_0)-t_0$ such that $ u-x_0-\max\{t^{1/\beta}(1-x_0)^{1-1/\beta}-t|t\in[v, \phi_{x_0}(u-x_0)]\}=u-x_0-\phi^{-1}_{x_0}(t_0)+t_0> v$. 
		\item  [($iii$)] If $t_0\le v$, then $\max\{t^{1/\beta}(1-x_0)^{1-1/\beta}-t|t\in[v, \phi_{x_0}(u-x_0)]\}=\phi^{-1}_{x_0}(v)-v$ such that $ u-x_0-\max\{t^{1/\beta}(1-x_0)^{1-1/\beta}-t|t\in[v, \phi_{x_0}(u-x_0)]\}=u-x_0-\phi^{-1}_{x_0}(v)+v\ge v$, where the inequality follows from $v\le \phi_{x_0}(u-x_0)$. 
\end{itemize}
Summarizing the above results,   when $(u,v)\in T_{L,\phi_{x_0}}$, we have  $A_{\Gamma_{\phi_{x_0}}}(u,v)=\min\big\{v, u-x_0-\max\{t^{1/\beta}(1-x_0)^{1-1/\beta}-t|t\in[v, \phi_{x_0}(u-x_0)]\}\big\}=v=\min\{u,v\}$. Consequently, $A_{\Gamma_{\phi_{x_0}}}(u,v)=M(u,v)$ for all $(u,v)\in[0,1]^2$, and thus it is the largest quasi-copula with the prescribed sub-track section $\Gamma_{\phi_{x_0}}(t)=t$.

\end{example}

\section{Properties of bound (quasi-)copulas with a prescribed sub-track section}

This section investigates several properties of the bound (quasi-)copulas $B_{\Gamma_{\phi_{x_0}}}$ and $A_{\Gamma_{\phi_{x_0}}}$ with the prescribed sub-track section $\Gamma_{\phi_{x_0}}$, including their comparison with bound (quasi-)copulas with a given track section or a sub-diagonal section, as well as their  statistical characteristic and concordance order. 

\subsection{Comparison with bound (quasi-)copulas with a given track section or a sub-diagonal section}

In this subsection, the results obtained in the preceding sections are compared with the works of \citet{De2019On} and \citet{2022Best}, who introduced the notions of Bertino copula  $B_{\Gamma_{\phi}}$ with the curvilinear section $\Gamma_{\phi}$, $\Gamma_{\phi}$-curvilinear quasi-copula $A_{\Gamma_{\phi}}$ and $\Gamma_{\phi}$-curvilinear copula $K_{\Gamma_{\phi}}$, as well as with the work of \citet{2008quasi-copulas}, who introduced the sub-diagonal quasi-copula, i.e., the quasi-copula with a given sub-diagonal section. 

Since every $x_0$-curve $\phi_{x_0}\in\Phi_{x_0}$ coincides with the  curve $\phi\in \Phi$ when  $x_0=0$, i.e., $\phi_{0}(t)=\phi(t)$ for $t\in[0,1]$, the following results can be obtained.
\begin{prop}
	
	(i) When $x_0=0$,  $A_{\Gamma_{\phi_0}}$   equals  the quasi-copula $G_{\Gamma_{\phi_0}}$ of type \eqref{eq:A copula0}, which reduces to  $\Gamma_{\phi}$-curvilinear quasi-copula $A_{\Gamma_{\phi}}$ 
	with the given track section $\Gamma_{\phi}(t)=\Gamma_{\phi_0}\big(\phi_0(t)\big)$, $t\in[0,1]$ along the curve $\phi=\phi_0$.
	
	(ii) 
	When $x_0=0$,  $B_{\Gamma_{\phi_0}}$ reduces to the Bertino copula  $B_{\Gamma_{\phi}}$ 
	with the given curvilinear section $\Gamma_{\phi}(t)=\Gamma_{\phi_0}\big(\phi_0(t)\big)$, $t\in[0,1]$  along the curve $\phi=\phi_0$.
	
	(iii) When $x_0=0$,  $K_{\Gamma_{\phi_0}}$ reduces to $\Gamma_{\phi}$-curvilinear copula $K_{\Gamma_{\phi}}$ 
	with the given track section $\Gamma_{\phi}(t)=\Gamma_{\phi_0}\big(\phi_0(t)\big)$, $t\in[0,1]$  along the curve $\phi=\phi_0$.
	
	(iv) When $x_0=0$, the functions 
	$A^*_{\Gamma_{\phi_0}}$, $B^*_{\Gamma_{\phi_0}}$ and $K^*_{\Gamma_{\phi_0}}$ are three families of  symmetric quasi-copulas with the given  track section $\Gamma_{\phi}(t)=\Gamma_{\phi_0}\big(\phi_0(t)\big)$, $t\in[0,1]$  along the curve $\phi=\phi_0$.
\end{prop}

\begin{proof}
	We only prove the statement ($i$), since the statements ($ii$), ($iii$) and ($iv$) can be verified by similar arguments.
	
	 When $x_0=0$, from the expressions of $A_{\Gamma_{\phi_0}}$ and $G_{\Gamma_{\phi_0}}$, it follows immediately that $A_{\Gamma_{\phi_0}}(u,v)=G_{\Gamma_{\phi_0}}(u,v)$ for all $(u,v)\in[0,1]^2$.
	
Furthermore, when $x_0=0$, 
	\begin{equation*}A_{\Gamma_{\phi_{0}}}(u,v)
		=\left\{\begin{array}{ll}
			\min\big\{u, v-\max\{t-\Gamma_{\phi_{0}}(t)|t\in[ \phi_{0}(u),v]\}\big\}, & \ {\rm if}\  v\ge \phi_{0}(u), \\ 
			\min\big\{v, u-\max\{\phi^{-1}_{0}(t)-\Gamma_{\phi_{0}}(t)|t\in[v, \phi_{0}(u)]\}\big\},&  \ {\rm if}\  v< \phi_{0}(u)  
		\end{array}\right.
	\end{equation*}	
	\begin{equation*}
		\ \ \ \ \ \ \ \	\ \ \ \ \ \ \ \	\ \ \ \ \ \ \ \	=\left\{\begin{array}{ll}
			\min\big\{u, v-\max\{\phi_0(t)-\Gamma_{\phi_{0}}\big(\phi_0(t)\big)|t\in[u, \phi^{-1}_{0}(v)]\}\big\}, & \ {\rm if}\  v\ge \phi_{0}(u), \\ 
			\min\big\{v, u-\max\big\{t-\Gamma_{\phi_{0}}\big(\phi_0(t)\big)|t\in[\phi^{-1}_{0}(v), u]\}\big\},&  \ {\rm if}\  v< \phi_{0}(u) 
		\end{array}\right.
	\end{equation*}
	\begin{equation*}
		\ \ \ \ \ \ \ 	\ \ \ \ \ \ \ 			=\left\{\begin{array}{ll}
			\min\big\{u, v-\max\{\phi(t)-\Gamma_{\phi}\big(t\big)|t\in[u, \phi^{-1}(v)]\}\big\}, & \ {\rm if}\  v\ge \phi(u), \\ 
			\min\big\{v, u-\max\big\{t-\Gamma_{\phi}\big(t\big)|t\in[\phi^{-1}(v), u]\}\big\},&  \ {\rm if}\  v< \phi(u) 
		\end{array}\right.
	\end{equation*}
	$$=A_{\Gamma_{\phi}}(u,v), \ \ \ \ \ \ \  \ \	\ \ \ \ \ \ \ 	\ \ \ \ \ \ \ 	\ \ \ \ \ \ \ 	\ \ \ \ \ \ \ 	\ \ \ \ \ \ \ 	\ \ \ \ \ \ \ 	\ \ \ \ \ \ \ 	\ \ \ \ \ \ \ 	\ \ \ \ \ \ \ 	\ \ \ \ \ \ \ $$
 where  $\Gamma_{\phi}(t)=\Gamma_{\phi_0}\big(\phi_0(t)\big)$ is the track section along the curve $\phi=\phi_0$.
\end{proof}

For a fixed  $x_0\in(0,1)$, when $\phi_{x_0}(t)=t$ for $t\in[0,1-x_0]$, we directly obtain the bounds for the class of quasi-copulas with a prescribed  sub-diagonal section $\delta_{x_0}$. 

\begin{prop}
	For any $x_0\in(0,1)$, when $\phi_{x_0}(t)=t$ for  $t\in[0,1-x_0]$, $B_{\Gamma_{\phi_{x_0}}}$ and $A_{\Gamma_{\phi_{x_0}}}$ are respectively the 
 smallest and largest quasi-copulas 
	whose sub-diagonal section at $x_0$ coincides with $\delta_{x_0}(t)=\Gamma_{\phi_{x_0}}(t)$ for  $t\in[0,1-x_0]$.
\end{prop}

\subsection{Statistical characteristic of the  bound copula $B_{\Gamma_{\phi_{x_0}}}$ with a prescribed sub-track section }

In the following theorem, we provide a characterization of the best-possible bound copula $B_{\Gamma_{\phi_{x_0}}}$ with the given sub-track section $\Gamma_{\phi_{x_0}}$, expressed in terms of the joint distribution
of $$\max\big\{\phi_{x_0}(\max\{U-x_0,0\}), V\big\}\ \ {\rm and}\ \  \min\big\{\phi_{x_0}(\max\{U-x_0,0\}), V\big\},$$
  where $U$ and $V$ are uniform $[0, 1]$ random variables.

\begin{prop}\label{thm:stastical B}
	Let  
	$U$ and $V$ be random variables uniformly distributed on $[0,1]$, whose joint distribution function is given by a copula  $B_{\Gamma_{\phi_{x_0}}}$ with the prescribed  sub-track section $\Gamma_{\phi_{x_0}}$. Then, for every $(u,v)\in D_{L,x_0}$, there exist $t_1, t_2\in[\min\{\phi_{x_0}(u-x_0), v\}$,
	$ \max\{\phi_{x_0}(u-x_0), v\}]$ such that
	$$\mathbb{P}\big(\min\big\{\phi_{x_0}(\max\{U-x_0,0\}), V\big\}\le \min\{\phi_{x_0}(u-x_0), v\}\ {\rm and}\  \max\big\{\phi_{x_0}(\max\{U-x_0,0\}), V\big\}>\max\{\phi_{x_0}(u-x_0), v\}\big)$$
	\begin{equation}\label{eq:stastical B}
		=\mathbb{P}\big(U\le \phi^{-1}_{x_0}(t_1)+x_0<\phi^{-1}_{x_0}(V)+x_0\big)+\mathbb{P}\big(\phi^{-1}_{x_0}(V)+x_0\le \phi^{-1}_{x_0}(t_2)+x_0<U\big),\end{equation}
and for every $(u,v)\in [0,1]^2\setminus D_{L,x_0}$, $\mathbb{P}(U\le u, V \le v )=\mathbb{P}(1-v<U\le u)=\mathbb{P}(1-u<V\le v)$. 
	\end{prop}
	
	\begin{proof}
Fix a point $(u,v)\in D_{L,x_0}$ with $\phi_{x_0}(u-x_0)\le v$, and choose  $t_1, t_2\in[\phi_{x_0}(u-x_0), v]$ such that $\hat\Gamma^{I}_{\phi_{x_0}}(t_1)$ and $\hat\Gamma^{II}_{\phi_{x_0}}(t_2)$ are the respective minimum values of $\hat\Gamma^{I}_{\phi_{x_0}}(t)$ and $\hat\Gamma^{II}_{\phi_{x_0}}(t)$ on that interval. Then
\begin{eqnarray}\label{eq:statistical B1}
	&&\mathbb{P}\big(\min\big\{\phi_{x_0}(\max\{U-x_0,0\}), V\big\}\le \min\{\phi_{x_0}(u-x_0), v\}\ {\rm and}\nonumber\\
&&	
	\ \ \ \ \ \ \max\big\{\phi_{x_0}(\max\{U-x_0,0\}), V\big\}>\max\{\phi_{x_0}(u-x_0), v\}\big)\nonumber \\
	&=&\mathbb{P}(U\le u,V>v)+\mathbb{P}\big(V\le \phi_{x_0}(u-x_0),U>\phi_{x_0}^{-1}(v)+x_0\big)\nonumber\\
	&=&u-B_{\Gamma_{\phi_{x_0}}}(u,v)+\phi_{x_0}(u-x_0)-B_{\Gamma_{\phi_{x_0}}}\big(\phi_{x_0}^{-1}(v)+x_0,\phi_{x_0}(u-x_0)\big)\nonumber\\
	&=&u-\big(u-x_0-\min\{\phi^{-1}_{x_0}(t)-\Gamma_{\phi_{x_0}}(t)|t\in[\phi_{x_0}(u-x_0),v]\}\big)\nonumber\\
	&&+\phi_{x_0}(u-x_0)-\big(\phi_{x_0}(u-x_0)-\min\{t-\Gamma_{\phi_{x_0}}(t)|t\in[\phi_{x_0}(u-x_0),v]\}\big)\nonumber\\
	&=&x_0+\phi^{-1}_{x_0}(t_1)-\Gamma_{\phi_{x_0}}(t_1)+t_2-\Gamma_{\phi_{x_0}}(t_2)\nonumber\\
	&=&\mathbb{P}\big(U\le \phi^{-1}_{x_0}(t_1)+x_0<\phi^{-1}_{x_0}(V)+x_0\big)+\mathbb{P}\big(\phi^{-1}_{x_0}(V)+x_0\le \phi^{-1}_{x_0}(t_2)+x_0<U\big).\nonumber
\end{eqnarray}
Consequently, the equality \eqref{eq:stastical B} is verified.  The case $\phi_{x_0}(u-x_0)> v$ in $D_{L,x_0}$ follows by a similar argument. For $(u,v)\in [0,1]^2\setminus D_{L,x_0}$, it follows from $B_{\Gamma_{\phi_{x_0}}}(u,v)=W(u,v)$  that $\mathbb{P}(U\le u, V \le v )=\max\{u+v-1,0\}=\mathbb{P}(1-v<U\le u)=\mathbb{P}(1-u<V\le v)$. 
\end{proof}

\begin{example}
\label{EX: M++}  	Let $x_0\in(0,1)$ and $\Gamma_{\phi_{x_0}}(t)=t\big(x_0+(1-x_0)(\frac{t}{1-x_0})^{1/\beta}\big)^{\alpha}$ be the sub-track section of the Cuadras-Aug\'{e} family of copulas  along the $x_0$-curve $\phi_{x_0}(t)=(1-x_0)\big(\frac{t}{1-x_0}\big)^{\beta}$, $t\in[0, 1-x_0]$, $\beta\ge 1$. From Theorems \ref{B+} and \ref{A}, we know that the copula $B_{\Gamma_{\phi_{x_0}}}$ with the sub-track section   $\Gamma_{\phi_{x_0}}(t)=t\big(x_0+(1-x_0)(\frac{t}{1-x_0})^{1/\beta}\big)^{\alpha}$,   expressed as 
\begin{equation*}
B_{\Gamma_{\phi_{x_0}}}(u,v)
=	\left\{\begin{array}{ll}
		v-\min\{t-t\big(x_0+ &  \\
		\ \ \ \ \ \ \ \ 
		(1-x_0)(\frac{t}{1-x_0})^{1/\beta}\big)^{\alpha}|t\in[v,(1-x_0)\big(\frac{u-x_0}{1-x_0}\big)^{\beta}] \}, &   {\rm if}\  (u,v)\in D_{L, x_0}  \ {\rm and}\  v<(1-x_0)\big(\frac{u-x_0}{1-x_0}\big)^{\beta},\\
	u-x_0-\min\{(1-x_0)(\frac{t}{1-x_0})^{1/\beta}-t\big(x_0+&\\
			\ \ \ \ \ \ \ \ 	(1-x_0)(\frac{t}{1-x_0})^{1/\beta}\big)^{\alpha}|t\in[(1-x_0)\big(\frac{u-x_0}{1-x_0}\big)^{\beta},v] \},&  {\rm if}\  (u,v)\in D_{L, x_0}\ {\rm and} \ v\ge (1-x_0)\big(\frac{u-x_0}{1-x_0}\big)^{\beta}, \\
		W(u,v),&  \ {\rm otherwise}, 
	\end{array}\right.
\end{equation*}
is the best-possible lower bound for the class of copulas as well as for the class of quasi-copulas sharing the common sub-track section $\Gamma_{\phi_{x_0}}(t)=t\big(x_0+(1-x_0)(\frac{t}{1-x_0})^{1/\beta}\big)^{\alpha}$, and the quasi-copula 
	\begin{equation*}A_{\Gamma_{\phi_{x_0}}}(u,v)=\left\{\begin{array}{ll}
		\min\big\{u, v-\max\{t-t\big(x_0+ & \\
\ \ \ \ \ \ \ \ \ \ \ \ \ \ \ \ \ \ \ \ \ \ \ \ \ \ \ \ 
(1-x_0)(\frac{t}{1-x_0})^{1/\beta}\big)^{\alpha}|t\in[(1-x_0)\big(\frac{\max\{u-x_0,0\}}{1-x_0}\big)^{\beta},v] \}\big\}, & \ {\rm if}\  (u,v)\in T_{U,\phi_{x_0}}, \\ 
		\min\big\{v, u-x_0-\max\{(1-x_0)(\frac{t}{1-x_0})^{1/\beta}-t\big(x_0+&\\ 
\ \ \ \ \ \ \ \ \ \ \ \ \ \ \ \ \ \ \ \ \ \ \ \ \ \ \ \ 		(1-x_0)(\frac{t}{1-x_0})^{1/\beta}\big)^{\alpha}|t\in[v,(1-x_0)\big(\frac{u-x_0}{1-x_0}\big)^{\beta}] \}\big\},&  \ {\rm if}\  (u,v)\in T_{L,\phi_{x_0}}  
	\end{array}\right.
\end{equation*}	
is the best-possible upper bound for the class of quasi-copulas sharing the common sub-track section $t\big(x_0+(1-x_0)(\frac{t}{1-x_0})^{1/\beta}\big)^{\alpha}$.

From Proposition \ref{thm:stastical B}, we also have  that, for any $\beta\ge 1$, if the joint distribution function of uniform [0,1] random variables $U$ and $V$ is  the bound copula $B_{\Gamma_{\phi_{x_0}}}$, then for each   $(u,v)\in D_{L,x_0}$,  there exist $t_1, t_2\in[\min\{(1-x_0)\big(\frac{u-x_0}{1-x_0}\big)^{\beta},v\}$,
$ \max\{(1-x_0)\big(\frac{u-x_0}{1-x_0}\big)^{\beta},v\}]$ such that the equality \eqref{eq:stastical B} holds,  and for each $(u,v)\in [0,1]^2\setminus D_{L,x_0}$, $\mathbb{P}(U\le u, V \le v )=\mathbb{P}(1-v<U\le u)=\mathbb{P}(1-u<V\le v)$. 
\end{example}

\subsection{Concordance order}

We begin by recalling the definition of  concordance order between two (quasi-)copulas. For two quasi-copulas (or  copulas) $A$ and $B$, if $A(u,v)\leq B(u,v)$ for all $(u,v)\in[0,1]^2$, then $B$ is said to more concordant than $A$, written $A\prec B$ (see \cite{Nelsen(2006)}, for further details). 
The following proposition shows that the concordance order of bound  (quasi-)copulas  $B_{\Gamma_{\phi_{x_0}}}$ and $A_{\Gamma_{\phi_{x_0}}}$ is determined by the pointwise ordering of their sub-track  sections.

\begin{prop}
\label{prop: co} 

($i$) 
Let $B_{\Gamma_{\phi_{x_0}}}$ and $B_{\tilde{\Gamma}_{\phi_{x_0}}}$ be two copulas of type \eqref{eq:Bx} with the sub-track sections $\Gamma_{\phi_{x_0}}$ and $\tilde{\Gamma}_{\phi_{x_0}}$, respectively.
Then $B_{\Gamma_{\phi_{x_0}}}\prec B_{\tilde{\Gamma}_{\phi_{x_0}}}$ if and only if\ \   $\Gamma_{\phi_{x_0}}(t)\le \tilde{\Gamma}_{\phi_{x_0}}(t)$ for all $t\in[0,1-x_0]$.

($ii$) 
Let $A_{\Gamma_{\phi_{x_0}}}$ and $A_{\tilde{\Gamma}_{\phi_{x_0}}}$ be two quasi-copulas of type \eqref{eq:G copula1*} with the sub-track sections $\Gamma_{\phi_{x_0}}$ and $\tilde{\Gamma}_{\phi_{x_0}}$, respectively. 
Then $A_{\Gamma_{\phi_{x_0}}}\prec A_{\tilde{\Gamma}_{\phi_{x_0}}}$ if and only if\ \  $\Gamma_{\phi_{x_0}}(t)\le \tilde{\Gamma}_{\phi_{x_0}}(t)$ for all $t\in[0,1-x_0]$.

\end{prop}

\begin{proof} \textit{(i)}. Note that when $(u,v)\in  D_{L,x_0}$ with $\phi_{x_0}(u-x_0)> v$, 
$-\min\{t-\Gamma_{\phi_{x_0}}(t)|t\in[v,\phi_{x_0}(u-x_0)]\}\le- \min\{t-\tilde{\Gamma}_{\phi_{x_0}}(t)|t\in[v,\phi_{x_0}(u-x_0)]\}$ if and only if $\Gamma_{\phi_{x_0}}(t)\le \tilde{\Gamma}_{\phi_{x_0}}(t)$ for all $t\in[0,1-x_0]$. Then, from the expression  \eqref{eq:Bx} of the copula $B_{\Gamma_{\phi_{x_0}}}$, it follows that, when $(u,v)\in  D_{L,x_0}$ with  $\phi_{x_0}(u-x_0)> v$, $B_{\Gamma_{\phi_{x_0}}}(u,v)\le  B_{\tilde{\Gamma}_{\phi_{x_0}}}(u,v)$ if and only if $\Gamma_{\phi_{x_0}}(t)\le \tilde{\Gamma}_{\phi_{x_0}}(t)$ for all $t\in[0,1-x_0]$. Similarly, when $(u,v)\in  D_{L,x_0}$ with $v\ge  \phi_{x_0}(u-x_0)$, $B_{\Gamma_{\phi_{x_0}}}(u,v)\le  B_{\tilde{\Gamma}_{\phi_{x_0}}}(u,v)$ if and only if $\Gamma_{\phi_{x_0}}(t)\le \tilde{\Gamma}_{\phi_{x_0}}(t)$ for all $t\in[0,1-x_0]$. Therefore, the statement ($i$) holds. 

\textit{(ii)}. Since the argument is completely analogous to that of the statement ($i$), it is omitted.
\end{proof}

\begin{example} For a fixed $x_0\in(0,1)$, 
consider the $x_0$-curve $\phi_{x_0}(t)=\frac{(1-x_0)t}{2(1-x_0)-t}$ and the  sub-track section $\Gamma^{\theta}_{\phi_{x_0}}(t)=\frac{t\big((2-x_0)t+x_0(1-x_0)\big)}{t+(1-x_0)\big(1-\theta(1-x_0-t)(1-t)\big)}$, $\theta\in[-1,1]$, $t\in[0,1-x_0]$. Note that for any fixed $t\in[0,1]$, the sub-track section $\Gamma^{\theta}_{\phi_{x_0}}(t)$ is nondecreasing with respect to $\theta$. Then from Proposition \ref{prop: co}, we know that for the  copula of type \eqref{eq:Bx}  with the sub-track section $\Gamma^{\theta}_{\phi_{x_0}}$, which is expressed as
\begin{equation*}
B_{\Gamma^{\theta}_{\phi_{x_0}}}(u,v)	=	\left\{\begin{array}{ll}
		v-\min\Big\{t-\frac{t\big((2-x_0)t+x_0(1-x_0)\big)}{t+(1-x_0)\big(1-\theta(1-x_0-t)(1-t)\big)}| & \\ 	 
	
\ \ \ \ \ \ \ \ \ \ \ \ \ \ \ \ \ \ \ \ \ \ \ \ \ \ \ \ \ \	t\in	[v,\frac{(1-x_0)(u-x_0)}{2(1-x_0)-(u-x_0)}] \Big\}, & (u,v)\in D_{L, x_0}  \ {\rm and}\  v<\frac{(1-x_0)(u-x_0)}{2(1-x_0)-(u-x_0)},\\
		u-x_0-\min\Big\{\frac{2(1-x_0)t}{1-x_0+t}-\frac{t\big((2-x_0)t+x_0(1-x_0)\big)}{t+(1-x_0)\big(1-\theta(1-x_0-t)(1-t)\big)}| & \\ 

\ \ \ \ \ \ \ \ \ \ \ \ \ \ \ \ \ \ \ \ \ \ \ \ \ \ \ \ \ \	t\in	[\frac{(1-x_0)(u-x_0)}{2(1-x_0)-(u-x_0)},v] \Big\},& (u,v)\in D_{L, x_0}\ {\rm and} \ v\ge \frac{(1-x_0)(u-x_0)}{2(1-x_0)-(u-x_0)}, \\
		W(u,v),&  \ {\rm otherwise}, 
	\end{array}\right.
\end{equation*}
$B_{\Gamma^{\theta_1}_{\phi_{x_0}}}\prec B_{\Gamma^{\theta_2}_{\phi_{x_0}}}$ if and only if $\theta_1\le \theta_2$. And for the quasi-copula 
$$A_{\Gamma^{\theta}_{\phi_{x_0}}}(u,v)
=\left\{\begin{array}{ll}
	\min\Big\{u, v-\max\{t-\frac{t\big((2-x_0)t+x_0(1-x_0)\big)}{t+(1-x_0)\big(1-\theta(1-x_0-t)(1-t)\big)}|t\in[\frac{(1-x_0)\max\{u-x_0,0\}}{2(1-x_0)-\max\{u-x_0,0\}},v] \}\Big\}, & (u,v)\in T_{U,\phi_{x_0}},\\
	\min\Big\{v, u-x_0-\max\{\frac{2(1-x_0)t}{1-x_0+t}-\frac{t\big((2-x_0)t+x_0(1-x_0)\big)}{t+(1-x_0)\big(1-\theta(1-x_0-t)(1-t)\big)}|t\in[v,\frac{(1-x_0)(u-x_0)}{2(1-x_0)-(u-x_0)}] \}\Big\},& (u,v)\in T_{L,\phi_{x_0}},
\end{array}\right.$$
it also holds that   $A_{\Gamma^{\theta_1}_{\phi_{x_0}}}\prec A_{\Gamma^{\theta_2}_{\phi_{x_0}}}$ if and only if $\theta_1\le \theta_2$.
\end{example}

	\section{\label{sec:Conclusion}Conclusions}
	
	In this paper, for a prescribed $x_0$-curve $\phi_{x_0}$ joining the points $(x_0,0)$ and $(1,1-x_0)$, where $0\le x_0<1$,  and a given sub-track section described by the $\phi_{x_0}$-function $\Gamma_{\phi_{x_0}}$, we propose 
	construction methods for quasi-copulas and copulas whose sub-track section along the $x_0$-curve $\phi_{x_0}$ coincides with $\Gamma_{\phi_{x_0}}$, and establish the best-possible bounds for the class of quasi-copulas sharing the common sub-track  section $\Gamma_{\phi_{x_0}}$. Using 
		three known (quasi-) copulas  $B_{\Gamma_{\phi}}$, $K_{\Gamma_{\phi}}$ and $A_{\Gamma_{\phi}}$ with the given track section $\Gamma_{\phi}$, together with the technique based on the $W$-ordinal sum, we construct (quasi-)copulas  $B_{\Gamma_{\phi_{x_0}}}$, $K_{\Gamma_{\phi_{x_0}}}$ and $G_{\Gamma_{\phi_{x_0}}}$ with the prescribed  sub-track section $\Gamma_{\phi_{x_0}}$, as well as the symmetric quasi-copulas $B^*_{\Gamma_{\phi_{x_0}}}$, $K^*_{\Gamma_{\phi_{x_0}}}$ and $G^*_{\Gamma_{\phi_{x_0}}}$ sharing the same sub-track section $\Gamma_{\phi_{x_0}}$. A symmetric copula $C_{\Gamma_{\phi_{x_0}}}$ with the prescribed sub-track section $\Gamma_{\phi_{x_0}}$ is also constructed. The function $B_{\Gamma_{\phi_{x_0}}}$ is shown to be the best-possible lower bound for the class of quasi-copulas sharing the common sub-track section $\Gamma_{\phi_{x_0}}$, and it is also the best-possible lower bound for the class of copulas with the same sub-track section. A quasi-copula $A_{\Gamma_{\phi_{x_0}}}$ is constructed and proved to be  the best-possible upper bound for the class of quasi-copulas with the sub-track section $\Gamma_{\phi_{x_0}}$. Several properties of these bound (quasi-)copulas $B_{\Gamma_{\phi_{x_0}}}$ and $A_{\Gamma_{\phi_{x_0}}}$ are also examined, including their comparison with bound (quasi-)copulas having a given track section or a sub-diagonal section, as well as their statistical characteristic and concordance order.

	\subsection*{Acknowledgement} The research was supported by the National Natural Science Foundation of China (Grants
	No. 12661104, 72271113) and the Natural Science Foundation of Jiangxi Province (Grants No. 20232ACB201003). All authors share first authorship and
	are listed alphabetically.

	\bibliographystyle{elsarticle-harv}
	\bibliography{curlinear}

\end{document}